\documentclass[11pt]{amsart}
\usepackage{graphicx}
\usepackage{amsmath}
\usepackage{amssymb}
\usepackage{amsfonts}
\usepackage{verbatim}
\usepackage{scalefnt}
\usepackage{multirow}
\usepackage{enumerate}
\usepackage{mathtools}
\usepackage{float}
\usepackage{enumitem}
\restylefloat{figure}
\usepackage[margin=3cm]{geometry}

\usepackage{pstricks,tikz}
\usetikzlibrary{patterns}
\usepackage{fancyhdr}
\usepackage{hyperref}

\DeclareMathOperator*{\rt}{{\rm RT}}

\begin{document}

\def\COMMENT#1{}

\newcommand{\case}[1]{\medskip\noindent{\bf Case #1} }
\newcommand{\step}[1]{\medskip\noindent{\bf Step #1} }
\newtheorem{problem}{Problem}
\newtheorem{theorem}{Theorem}
\newtheorem{lemma}[theorem]{Lemma}
\newtheorem{proposition}[theorem]{Proposition}
\newtheorem{corollary}[theorem]{Corollary}
\newtheorem{conjecture}[theorem]{Conjecture}
\newtheorem{claim}[theorem]{Claim}
\newtheorem{definition}[theorem]{Definition}
\newtheorem*{definition*}{Definition}
\newtheorem{fact}[theorem]{Fact}
\newtheorem{observation}[theorem]{Observation}
\newtheorem{question}[theorem]{Question}
\newtheorem{remark}[theorem]{Remark}

\numberwithin{equation}{section}
\numberwithin{theorem}{section}

\def\eps{{\varepsilon}}
\renewcommand{\epsilon}{\varepsilon}
\newcommand{\cP}{\mathcal{P}}
\newcommand{\cT}{\mathcal{T}}
\newcommand{\cL}{\mathcal{L}}
\newcommand{\ex}{\mathbb{E}}
\newcommand{\eul}{e}
\newcommand{\pr}{\mathbb{P}}
\newcommand{\mincr}{\delta^{\rm cr}}
\newcommand{\phiind}{\phi^{\rm ind}}

\title[Two conjectures in Ramsey-Tur\'an theory]{Two conjectures in Ramsey-Tur\'an theory}
\author{Jaehoon Kim, Younjin Kim and Hong Liu}
\thanks{J.K.\ was supported by ERC grant~306349; Y.K.\ was supported by Basic Science Research Program through the National Research Foundation of Korea(NRF) funded by the Ministry of Education (2017R1A6A3A04005963); and H.L.\ was supported by the Leverhulme Trust Early Career Fellowship~ECF-2016-523.}

\begin{abstract}
	Given graphs $H_1,\ldots, H_k$, a graph $G$ is $(H_1,\ldots, H_k)$-free if there is a $k$-edge-colouring $\phi:E(G)\rightarrow [k]$ with no 
monochromatic copy of	$H_i$ with edges of colour $i$ for each $i\in[k]$. Fix a function $f(n)$, the Ramsey-Tur\'an function $\rt(n,H_1,\ldots,H_k,f(n))$ is the maximum number of edges in an $n$-vertex $(H_1,\ldots,H_k)$-free graph with independence number at most $f(n)$. We determine $\rt(n,K_3,K_s,\delta n)$ for $s\in\{3,4,5\}$ and sufficiently small $\delta$, confirming a conjecture of Erd\H os and S\'os from 1979. It is known that $\rt(n,K_8,f(n))$ has a phase transition at $f(n)=\Theta(\sqrt{n\log n})$. However, the values of $\rt(n,K_8, o(\sqrt{n\log n}))$ was not known. We determined this value
by proving $\rt(n,K_8,o(\sqrt{n\log n}))=\frac{n^2}{4}+o(n^2)$, answering a question of Balogh, Hu and Simonovits. The proofs utilise, among others, dependent random choice and results from graph packings.
\end{abstract}

\date{\today}
\maketitle

\section{Introduction and results}
Tur\'an's theorem~\cite{Turan} states that among all $n$-vertex $K_{s+1}$-free graphs, the balanced complete $s$-partite graph, now so-called $s$-partite \emph{Tur\'an graph} $T_s(n)$, has the largest size, where the size of a graph is the number of edges in a graph. Notice that these Tur\'an graphs have rigid structures, in particular, there are independent sets of size linear in $n$. It is then natural to ask for the size of an $n$-vertex $K_{s+1}$-free graph without these rigid structures, i.e. graphs with additional contstraints on their independence number. Such problems, first introduced by S\'os~\cite{ESos-1} in 1969, are the substance of the Ramsey-Tur\'an theory. Formally, given a graph $H$ and natural numbers $m,n\in \mathbb{N}$, the \emph{Ramsey-Tur\'{a}n number}, denoted by $\rt(n,H,m)$ is the maximum number of edges in an $n$-vertex $H$-free graph $G$ with $\alpha(G)\le m$ can have. Motivated by above reasons, the most classical case is when $m$ is sublinear in $n$, i.e.~$m=o(n)$.
\begin{definition*}
	Given a graph $H$ and $\delta\in (0,1)$, let 
	\begin{equation}\label{eq-defn}
	\varrho(H,\delta):=\lim_{n\rightarrow\infty}\frac{\rt(n,H,\delta n)}{n^2}\quad \mbox{ and } \quad \varrho(H):=\lim_{\delta \rightarrow 0}\varrho(H,\delta).
	\end{equation}
   We write
	$$\rt(n,H,o(n))=\varrho(H)\cdot n^2+o(n^2).$$
\end{definition*}

We call $\varrho(H)$ the \emph{Ramsey-Tur\'an density} of $H$.
The Ramsey-Tur\'an density of cliques are well-understood. For odd cliques, Erd\H os and S\'os~\cite{ESos-1} proved that $\varrho(K_{2s+1})=\frac{1}{2}(\frac{s-1}{s})$ for all $s\ge 1$. The problem for even cliques is much harder. Szemer\'edi~\cite{Sz-K4} first showed that $\varrho(K_4)\le \frac{1}{8}$.
However no lower bound on $\varrho(K_4)$ was known until Bollob\'as and Erd\H os~\cite{Bollobas-E} provided a matching lower bound using an ingenious geometric construction, showing that $\varrho(K_4)=\frac{1}{8}$. Finally, Erd\H os, Hajnal, S\'os and Szemer\'edi~\cite{EHSosSz} determined the Ramsey-Tur\'an density for all even cliques, proving that $\varrho(K_{2s})=\frac{1}{2}(\frac{3s-5}{3s-2})$ for all $s\ge 2$. 

While $\varrho(H)$ shows only the limit value, $\varrho(H,\delta)$ captures the transition behaviours of Ramsey-Tur\'an number more accurately when independence number drops to $o(n)$. 
Capturing this more subtle behaviour, Fox, Loh and Zhao~\cite{FLZ} proved that $\varrho(K_4, \delta)=\frac{1}{8} + \Theta(\delta)$. Building on Fox-Loh-Zhao's work, L\"uders and Reiher~\cite{LR} have very 
recently showed that, surprisingly, there is a precise formula for $\varrho(H,\delta)$ for all cliques and sufficiently small $\delta$: for all $s\ge 2$, $\varrho(K_{2s-1},\delta)=\frac{1}{2}(\frac{s-2}{s-1}+\delta)$, while $\varrho(K_{2s},\delta)=\frac{1}{2}(\frac{3s-5}{3s-2}+\delta-\delta^2)$. Inspired by L\"uders and Reiher's work, one of our results concerns the multicolour extension of this result. For more literature on Ramsey-Tur\'an theory, we refer the readers to a survey of Simonovits and S\'os~\cite{Sim-Sos-Survey}. See also~\cite{BLSh, BMMM,BMS} for more recent results on variants of Ramsey-Tur\'an problem.

\subsection{Multicolour Ramsey-Tur\'an problem}
Given graphs $H_1,\ldots,H_k$, we say that a graph $G$ is \emph{$(H_1,\ldots, H_k)$-free} if there exists an edge colouring $\phi:E(G)\rightarrow [k]$ such that for each $i\in [k]$, the spanning subgraph with all edges of colour $i$ is $H_i$-free. Let $\rt(n,H_1,\ldots,H_k,m)$ be the maximum size of an $n$-vertex $(H_1,\ldots,H_k)$-free graph with independence number at most $m$, and define $\varrho(H_1,\ldots, H_k,\delta)$ and $\varrho(H_1,\ldots,H_k)$ analogous to~\eqref{eq-defn}. Erd\H os, Hajnal, Simonovits, S\'os and Szemer\'edi~\cite{EHSSSz} proved that the multicolour Ramsey-Tur\'an density for cliques is determined by certain weighted Ramsey numbers (see Definition~5 and Theorem~2 in~\cite{EHSSSz} for more details). Determining the actual values of $\varrho(K_{s_1},\ldots,K_{s_k})$ turns out to be very difficult. Only sporadic cases are known~\cite{EHSSSz}: $\varrho(K_3,K_3)=\frac{1}{4}$, $\varrho(K_3,K_4)=\frac{1}{3}$, $\varrho(K_3,K_5)=\frac{2}{5}$ and $\varrho(K_4,K_4)=\frac{11}{28}$. Even 
  2-coloured triangle versus a clique case, i.e.~determining $\varrho(K_3,K_s)$, remains open. Recall that the Ramsey number $R(s,t)$ is the minimum integer $N$ such that every blue/red colouring of $K_N$ contains either a blue $K_s$ or a red $K_t$. Erd\H os, Hajnal, Simonovits, S\'os and Szemer\'edi~\cite{EHSSSz} conjectured for all $s\ge 2$, $\varrho(K_3,K_{2s-1})=\frac{1}{2}\left(1-\frac{1}{R(3,s)-1}\right)$.

Capturing more subtle behaviours of multicolour Ramsey-Tur\'an number, Erd\H os and S\'os~\cite{ES} proved in 1979 that $\varrho(K_3,K_3,\delta)=\frac{1}{4}+\Theta(\delta)$ and conjectured that for sufficiently small $\delta$, there exists $c>0$ such that $\varrho(K_3,K_3,\delta)=\frac{1}{4}+c\delta$. In the following theorem, we determine the exact value of $\varrho(K_3, K_3,\delta)$ for all small $\delta>0$, thus confirming the conjecture of Erd\H{o}s and S\'os. Furthermore, we also determine the exact values of $\varrho(K_3,K_4,\delta)$ and $\varrho(K_3,K_5,\delta)$. We remark that $\varrho(K_3,K_4,\delta)$ behaves quite differently from  $\varrho(K_3,K_s,\delta)$ with $s\in\{3,5\}$. The extremal graph achieving the value of $\varrho(K_3,K_3,\delta)$ (resp. $\varrho(K_3,K_5,\delta)$) comes from taking the union of $T_2(n)$ (resp. $T_5(n)$) and $F^*$, certain almost $\delta n$-regular $K_3$-free graph with independence number at most $\delta n$. It turns out that the natural lower bound from the union of $T_3(n)$ and $F^*$ is not optimal for $\varrho(K_3,K_4,\delta)$.

\begin{theorem}\label{thm-K3Ks}
	For sufficiently small $\delta>0$, we have
	\begin{itemize}
	\item $\varrho(K_3,K_3,\delta )=\frac{1}{4}+\frac{\delta}{2}$;
	
	\item $\varrho(K_3,K_4,\delta )=\frac{1}{3}+\frac{\delta }{2}+\frac{3\delta^2}{2}$;
	
	\item $\varrho(K_3,K_5,\delta )=\frac{2}{5}+\frac{\delta }{2}$;
    \end{itemize}
\end{theorem}

%
%
%

We can see that the 2-colour Ramsey-Tur\'an number $\varrho(K_3,K_s,\delta)$ shares some similarity with the single-colour problem $\varrho(K_s,\delta)$ as they both have an extra quadratic term when $s$ is even. However, the single-colour Ramsey-Tur\'an number has the same quadratic term for all even $s$. This is not the case for the 2-colour Ramsey-Tur\'an number due to its relation to Ramsey number $R(3,\lceil s/2\rceil)$. Indeed, we give a construction showing that $$\varrho(K_3,K_6,\delta )\ge\frac{5}{12}+\frac{\delta }{2}+2\delta^2.$$ 
We conjecture that the equality above holds (see concluding remark for more details).

In the following theorem,
we determine Ramsey-Tur\'an numbers for $(K_3,K_s)$ for all $s\ge 3$ when the independence number condition is slightly more strict than sublinear, providing evidence towards the Erd\H os-Hajnal-Simonovits-S\'os-Szemer\'edi conjecture. Let $\omega(n)$ be a function growing to infinity arbitrarily slowly as $n\rightarrow \infty$. For each integer $s\ge 2$, define
\begin{equation}\label{eq-defn-sublinear}
	g^{\omega}_s(n):=\frac{n}{e^{\omega(n)\cdot (\log n)^{1-1/s}}}.
\end{equation}
We omit $\omega$ and write $g_{s}(n)$ whenever the result holds for any function $\omega(n)$ growing to infinity. Note that $n \gg g_s(n) \gg n^{1-\eps}$, for any $\eps>0$.
\begin{theorem}\label{thm-3q-sm-ind}
	For all $s\ge 2$, we have 
	\begin{itemize}
		\item $\rt(n,K_3,K_{2s-1},g_s(n))=\frac{1}{2} \left(1-\frac{1}{R(3,s)-1}\right)n^2+o(n^2)$; and
	
	\item$\rt(n,K_3,K_{2s},g_s(n))=\frac{1}{2} \left(1-\frac{1}{R(3,s)}\right)n^2+o(n^2)$.
\end{itemize}

\end{theorem}

\subsection{Phase transition}
Our next result concerns phase transitions of the single-coloured Ramsey-Tur\'an number. A graph $H$ has \emph{Ramsey-Tur\'an phase transition at $f$} if $$\lim_{n\rightarrow\infty}\frac{\rt(n,H,f(n))-\rt(n,H,o(f(n)))}{n^2}>0,$$ 
where $\rt(n,H,o(f(n))=\lim_{\delta \rightarrow 0}\rt(n,H,\delta\cdot f(n))$. In other words, a slightly stronger upper bound on the independence number, $o(f(n))$ instead of $f(n)$, would result in a drop at the maximum possible edge-density of an $H$-free graph (see~\cite{BHS} for more details).

From odd cliques, the result of Erd\H os-S\'os~\cite{ESos-1} shows that $K_{2s+1}$, with $s\ge 1$, has its first phase transition at $f(n)=n$, where the density drops from $\frac{1}{2}(\frac{2s-1}{2s})$ to $\frac{1}{2}(\frac{s-1}{s})$. In fact, $\rt(n,K_{2s+1},c\sqrt{n\log n})=\frac{1}{2}(\frac{s-1}{s}+o(1))n^2$ for $c>2/\sqrt{s}$. Balogh, Hu and Simonovits (\cite{BHS}, Theorem~2.7) proved that $\rt(n,K_{2s+1},o(\sqrt{n\log n}))\le\frac{1}{2}(\frac{2s-3}{2s}+o(1))n^2$, showing that the second phase transition happens at $f(n)=\sqrt{n\log n}$ (around the inverse function of $R(3,n)$). Erd\H os and S\'os~\cite{ESos-1} asked whether $\rt(n,K_5,o(\sqrt{n}))=o(n^2)$. Sudakov~\cite{Sudakov-K4} showed that it is true if a slightly stronger bound is imposed on the independence number: $\rt(n,K_5,g_2(\sqrt{n}))=o(n^2)$. Later, Balogh, Hu, Simonovits~\cite{BHS} answered Erd\H os and S\'os's question in a stronger form, showing that: $\rt(n,K_5,o(\sqrt{n\log n}))=o(n^2)$.

The situation for even cliques, $K_{2s}$ with $s\ge 2$, is again less clear apart from the first phase transition at $f(n)=n$ as shown by Erd\H os-Hajnal-Simonovits-S\'os-Szemer\'edi~\cite{EHSSSz}, where  the density decreases from $\frac{1}{2}(\frac{2s-2}{2s-1})$ to $\frac{1}{2}(\frac{3s-5}{3s-2})$. Extending a result of Sudakov~\cite{Sudakov-K4}, Balogh, Hu, Simonovits (\cite{BHS}, Theorem~3.4) showed that : $\rt(n,K_{2s},f(n))=\frac{1}{2}(\frac{s-2}{s-1}+o(1))n^2$ for any $c\sqrt{n\log n}<f(n)\le g_s(n)$ where $c>2/\sqrt{s-1}$; while Fox, Loh and Zhao~\cite{FLZ} showed that $\rt(n,K_{2s},g^*(n))=\frac{1}{2}(\frac{3s-5}{3s-2}+o(1))n^2$, where $g^*(n):=ne^{-o\left(\sqrt{\frac{\log n}{\log\log n}}\right)}$. Thus, the second phase transition for $K_{2s}$ happens somewhere in the small window between $g^*(n)$ and $g_s(n)$. The third phase transition for even cliques occurs at $f(n)=\sqrt{n\log n}$, but not a single extremal density is known except the trivial case of $K_4$. For example, $\rt(n,K_6,o(\sqrt{n\log n}))\le \frac{n^2}{6}+o(n^2)$ and we do not know whether it is $o(n^2)$. For $K_8$, Balogh, Hu and Simonovits~\cite{BHS} showed that  
\begin{itemize}
	\item $\rt(n,K_8, c\sqrt{n\log n})=\frac{n^2}{3}+o(n^2)$ for $c>2/\sqrt{3}$;
	\item $\frac{2n^2}{7}+o(n^2)\ge \rt(n,K_8,o(\sqrt{n\log n})) \ge \rt(n,K_7,o(\sqrt{n\log n}))=\frac{n^2}{4}+o(n^2)$; 
	\item $\rt(n,K_8,g_2(\sqrt{n}))=\frac{n^2}{4}+o(n^2)$,
\end{itemize}
and raised the question of whether $\rt(n,K_8,o(\sqrt{n\log n}))=\rt(n,K_7,o(\sqrt{n\log n}))$. So the Ramsey-Tur\'an density for $K_8$ drops from $1/3$ to at most $2/7$ around $\sqrt{n\log n}$. It is not clear when in between $o(\sqrt{n\log n})$ and $g_2(\sqrt{n})$, it drops to $1/4$. In the following theorem, we close this gap, proving that $\rt(n,K_8,o(\sqrt{n\log n}))=\frac{n^2}{4}+o(n^2)$. This answers Balogh-Hu-Simonovits's question positively and provides the first exact value of nontrivial extremal density for the third phase transition of an even clique.

\begin{theorem}\label{thm-RTK8}
	For any $\gamma>0$, there exists $\delta>0$ such that the following holds. Let $G$ be an $n$-vertex $K_8$-free graph with $\alpha(G)\le \delta\sqrt{n\log n}$. Then $e(G)\le \frac{n^2}{4}+\gamma n^2$.
\end{theorem}

\medskip

\noindent\textbf{Organisation of the paper.} In Section~\ref{sec-prelim}, we give preliminaries necessary for the proofs. Then we present the proofs of Theorem~\ref{thm-RTK8} in Section~\ref{sec-K8}, and the lower bounds in Theorem~\ref{thm-K3Ks} in Section~\ref{sec-lowerbound-2col}.
We then prove the upper bounds in Theorem~\ref{thm-K3Ks} in Sections~\ref{sec-K3Ks} and~\ref{sec-K3K3}. The proof of Theorem~\ref{thm-3q-sm-ind} will be given in Section~\ref{sec-3q}. Finally in Section~\ref{sec-remarks}, we make some concluding remarks.
\section{Preliminaries}\label{sec-prelim}
In this section, we introduce some notation, tools and lemmas. Denote $[q]:=\{1,2,\ldots,q\}$, $[p,q]:=\{p,p+1,\ldots, q\}$, and ${X\choose i}$ (resp.\ ${X\choose \le i}$) denotes the set of all subsets of a set $X$ of size $i$ (resp.\ at most $i$).  We may abbreviate a singleton $\{x\}$ (resp.\ a pair $\{x,y\}$) as $x$ (resp.\ $xy$). If we claim that a result holds whenever $0<b\ll a\ll 1$, this means that there are a constant $a_0\in (0,1)$ and a non-decreasing function $f : (0, 1) \rightarrow (0, 1)$ (that may depend on any previously defined constants or functions) such that the result holds for all $a,b\in (0,1)$ with $a\le a_0$ and $ b \leq f(a)$. We write $a=b\pm c$ if $b-c\leq a\leq b+c$. We may omit floors and ceilings when they are not essential.

Let $G=(V,E)$ be a graph and $A,B,V_1,\ldots,V_p\subseteq V$. Denote by $\overline{A}:=V\setminus A$ the complement of $A$. Let $G[A]:=(A,\{xy\in E: x,y\in A\})$ denote the induced subgraph of $G$ on $A$, and denote by $N(A,B)$ the common neighbourhood of $A$ in $B$. Write $N(v,B)$ instead of $N(\{v\}, B)$, and $d(v,B)=|N(v,B)|$. Denote by $G[V_1,\ldots, V_p]$ the $p$-partite subgraph of $G$ induced by $p$-partition $V_1\cup\dots \cup V_p$. We say that a partition $U_1\cup\ldots\cup U_p$ of $V$ is a \emph{max-cut $p$-partition} of $G$ if $e(G[U_1,\ldots,U_p])$ is maximised among all $p$-partition of $V$. Denote by $\mincr(G[V_1,\ldots, V_p]):=\min_{ij\in{[p]\choose 2},~v\in V_i}d(v,V_j)$ the \emph{minimum crossing degree} of $G$ with respect to the partition $V_1\cup \ldots\cup V_p$. 
For each $n,p\in [n]$, $T_p(n)$ denotes the $n$-vertex Tur\'an graph, which is the $n$-vertex complete $p$-partite graph such that each partite sets has size either $\lfloor n/p \rfloor$ or $\lceil n/p \rceil$.
For two $n$-vertex graphs $G$ and $H$, we define 
$|G\triangle H|$ be the minimum number $N=N_1+N_2$ such that we can obtain a graph isomorphic to $H$ after deleting $N_1$ edges from $G$ and adding $N_2$ edges to $G$.

Given $\phi: E(G)\rightarrow [k]$, throughout the paper, for each $i\in [k]$, we will always denote $G_i$ the spanning subgraph of $G$ induced by all edges of colour $i$. We say that $\phi$, and also $G$, is \emph{$(K_{s_1},\ldots,K_{s_k})$-free} if $G_i$ is $K_{s_i}$-free for each $i\in [k]$. We will write $\phi(A,B)=i$ if $\phi(e)=i$ for all $e\in E(G[A,B])$, and write $\phi(v,B)$ instead of $\phi(\{v\},B)$. If $\phi$ is also defined on $V(G)$, we write $\phi(A)=i$ if $\phi(v)=i$ for all $v\in A$. The following result will be useful.

\begin{theorem}[\cite{GK}]\label{thm-disj-k3}
	Let $G$ be an $n$-vertex $K_4$-free graph with $e(G)\ge n^2/4+t$. Then $G$ contains at least $t$ edge-disjoint triangles.
\end{theorem}

Given $d,n\in\mathbb{N}$, denote by $F(n,d)$ an $n$-vertex $d$-regular triangle-free graph with $\alpha(G)=d$. Let $\mathcal{B}\subseteq (0,1)$ consists of all the rationals $\delta$ for which there exists some $F(n,d)$ with $d/n=\delta$. We will use a result of Brandt~\cite{Brandt}, which states that $\mathcal{B}$ is dense in $(0,1/3)$, in the following form.

\begin{theorem}[\cite{Brandt}]\label{thm-Brandt}
	For any $0<\eta, \delta<1/3$, there exists $n_0>0$ such that the following holds for all $n\ge n_0$. For some $d\in [(\delta-\eta)n,\delta n]$,
	there exists a graph $F(n,d)$.
\end{theorem}

The following is a result of F\"uredi proving stability of $K_{p+1}$-free graphs.
\begin{theorem}[\cite{Furedi}]\label{thm-furedi-stab}
	Suppose that $t\in \mathbb{N}$ and $G$ is an $n$-vertex $K_{p+1}$-free with $e(G)\ge e(T_{n,p})-t$. Then there exists $v_1,\ldots,v_p$ such that $e(K_{v_1,\ldots,v_p})\ge e(T_{n,p})-2t$ and
	$|G\triangle K_{v_1,\ldots,v_p}|\le 3t$. Consequently, $v_i=n/p\pm 2\sqrt{t}$ for all $i\in [p]$ and $|G\triangle T_p(n)|= O(\sqrt{t}n)$.
\end{theorem}

The following theorem follows from Shearer's bound on Ramsey number $R(3,k)\le (1+o(1))k^2/\log k$ (see also~\cite{BK,DJPR,PGM} for more recent development on $R(3,k)$).
\begin{theorem}\label{thm: Shearer}\cite{Shearer}
There exists $k_0\in \mathbb{N}$ such that for all $k\geq k_0$, any graph on 
at least $2k^2/\log{k}$ vertices contains either a triangle or an independent set of size $k$.
\end{theorem}

We will make use of the multicolour version of the Szemer\'edi Regularity Lemma (see, for example,~\cite[Theorem 1.18]{KomlosSimonovits96}). We introduced the relevant definitions. Let $X,Y\subseteq V(G)$ be disjoint non-empty sets of vertices in a graph $G$. The \emph{density} of $(X,Y)$ is    
$d_{G}(X,Y):=\frac{e(G[X,Y])}{|X|\,|Y|}$. For $\eps>0$, the pair $(X,Y)$ is {\it $\eps$-regular} in $G$ if for every pair of subsets $X'\subseteq X$ and $Y'\subseteq Y$ with $|X'|\ge\eps|X|$ and $|Y'|\ge\eps|Y|$, we have $|d_{G}(X,Y)-d_{G}(X',Y')|\le \eps$. Additionally, if $d_{G}(X,Y)\ge\gamma$, for some $\gamma>0$, we say that $(X,Y)$ is {\it $(\eps,\gamma)$-regular}. A partition $V(G)=V_1\cup\ldots V_m$ is an \emph{$\eps$-regular partition} of a $k$-edge-coloured graph $G$ if
\begin{enumerate}
	\item for all $ij\in {[m]\choose 2}$, $\big|\,|V_i|-|V_j|\,\big|\le 1$; 
	\item  for each $i\in [m]$, all but at most $\epsilon m$ choices of $j\in [m]$ satisfies that the pair $(V_i,V_j)$ is $\eps$-regular in $G_{\ell}$ for each colour $\ell \in [k]$.
\end{enumerate}

\begin{lemma}[Multicolour Regularity Lemma~\cite{KomlosSimonovits96}]\label{lm:RL}
Suppose $0< 1/M' \ll \epsilon, 1/M \ll 1/k \leq 1$ and $n\geq M$.
Suppose $G$ is an $n$-vertex $k$-edge-coloured graph  and $U_1\cup U_2$ is a partition of $V(G)$. Then there exists an $\eps$-regular partition $V(G)=V_1\cup \ldots\cup V_m$ with $M\le m\le M'$ such that for each $i\in [m]$ we have either $V_i\subseteq U_1$ or $V_i\subseteq U_2$.
\end{lemma}

Given $\eps,\gamma>0$, a graph $G$, a colouring $\phi:E(G)\rightarrow [k]$ and a  partition $V(G)=V_1\cup\dots\cup V_m$, define the \emph{reduced graph} $R:=R(\eps,\gamma,\phi,(V_i)_{i=1}^m)$ of order $m$ as follows:
$V(R)=[m]$ and $ij\in E(R)$ if $(V_i,V_j)$ is (i) $\eps$-regular with respect to $G_{p}$ for every $p\in[k]$ and (ii) $d(G_q[V_i,V_j])\ge \gamma$ for some $q\in[k]$. For brevity, we may omit $\phi$ or $(V_i)_{i=1}^r$ in the notation when these are clear. It is easy to see that we have
\begin{align}\label{eq: size original reduced}
e(G)\le e(R)\cdot \left(\frac{n}{m}\right)^2 + \frac{n^2}{m} +k\gamma n^2.
\end{align}
Given a graph $R$ and $s\in\mathbb{N}$, let $R(s)$ be the graph obtained by replacing every vertex of $R$ with an independent set of size $s$ and replacing every edge of $R$ with $K_{s,s}$. The following lemmas provide some useful properties related to regular partitions.
\begin{lemma}[Counting Lemma, Theorem 2.1 in~\cite{KomlosSimonovits96}]\label{lem-embedding}
Suppose $0< 1/n  \ll \epsilon \ll \gamma, 1/h \leq 1$.
Suppose that $H$ is an $h$-vertex graph and $R$ is a graph such that 
$H\subseteq R(h)$. If $G$ is a graph obtained by replacing every vertex of $R$ with an independent set of size $n$ and replacing every edge of $R$ with an $(\epsilon,\gamma)$-regular pair, then $G$ contains at least $(\gamma/2)^{e(H)}N^{|V(H)|}$ copies of $H$.
\end{lemma}

\begin{lemma}[Slicing Lemma, Fact 1.5 in~\cite{KomlosSimonovits96}]\label{lem-slicing}
Let $\eps < \alpha,\gamma, 1/2$. Suppose that $(A,B)$ is an $(\eps,\gamma)$-regular pair in a graph $G$. 
If $A'\subseteq A$ and $B'\subseteq B$ satisfies $|A'|\ge\alpha|A|$ and $|B'|\ge\alpha|B|$, then $(A',B')$ is an $(\eps', \gamma-\eps)$-regular pair in $G$, where $\eps':=\max\{\eps/\alpha,2\eps\}$.
\end{lemma}

\begin{lemma}[Claim 7.1 in~\cite{BHS} with $p=2$]\label{lem-drc-j}
For a given function $g_s$ as in~\eqref{eq-defn-sublinear},
suppose $0< 1/n \ll 1/m, \eps\ll\gamma<1$. Suppose that $G$ is an $n$-vertex graph with $\alpha(G)\leq g_s(n)$ and $V_1\cup \dots \cup V_m$ is an $\eps$-regular partition and $R=R(\eps,\gamma)$ is a corresponding reduced graph.
If $K_s\subseteq R$, then we have $K_{2s}\subseteq G$.
\end{lemma}

The following lemma will be useful to guarantee a certain minimum degree condition in a dense graph.

\begin{lemma}\label{lem: min degree}
Suppose $0<1/n \ll \epsilon \ll d \leq 1$.
Suppose that $G$ is an $n$-vertex graph with $e(G)\geq (d+\epsilon) n^2/2$. Then $G$ contains an $n'$-vertex subgraph $G'$ with $n'\geq \epsilon^{1/2} n/2$ such that $e(G') \geq (d n'^2 + \epsilon n^2-d(n-n'))/2$ and $\delta(G')\geq d n'$.
\end{lemma}
\begin{proof}
Obtain a sequence of graphs $G_n:=G, G_{n-1},\ldots$ as follows. For each $i\le n$, if there exists a vertex $v_i\in V(G_i)$ with $d_{G_i}(v_i)\le di$, then set $G_{i-1}:=G_i\setminus\{v_i\}$. Let $G_{n'}$ be the final graph. Then we have
	$$ n'^2\ge 2e(G_{n'})> (d+\epsilon)n^2- \sum_{i = n'+1}^{n} 2d i 
	=  \epsilon n^2 + dn'^2 - dn + dn'.$$
This implies $n'\geq \epsilon^{1/2}n/2$, thus  proving the lemma.
\end{proof}

\section{Proof of Theorem~\ref{thm-RTK8}}\label{sec-K8}
We need first the following variation of dependent random choice lemma. For more on the dependent random choice method, we refer the readers to a survey of Fox and Sudakov~\cite{FS}.
\begin{lemma}\label{lem-drc-v}
	Let $0<\gamma<1$ and $G$ be a 3-partite graph with vertex partition $Z_1\cup Z_2\cup Z_3$ such that $|Z_i|=n$ for each $i\in[3]$. 
If $d(v,Z_i)\ge \gamma n$ holds for all $v\in Z_1$ and $i\in\{2,3\}$, then there exists a set $S\subseteq Z_1$ of size $\frac{1}{2}n^{2/3}$ such that every pair of vertices $P\in \binom{S}{2}$ satisfies $|N(P,Z_i)|\ge \gamma^9n$ for each $i\in\{2,3\}$. 
\end{lemma}
\begin{proof}
	Set $q:=\frac{\log n}{6\log(1/\gamma)}$. For each $i\in\{2,3\}$, pick $q$ vertices in $Z_i$ uniformly at random with repetition and denote by $Q_i$ the set of chosen vertices. We call a pair $P\in{Z_1\choose 2}$ \emph{bad} if there exists $i\in\{2,3\}$ such that $|N(P,Z_i)|<\gamma^9n$. Define $S':=N(Q_2\cup Q_3, Z_1)$, and define a random variable $X$ as the number of bad pairs in $S'$. Note that for each bad pair $P\in{Z_1\choose 2}$, we have
	$$\pr[P\subseteq S']=\pr[Q_2\cup Q_3\subseteq N(P)]=\prod_{i=2}^{3}\left(\frac{|N(P,Z_i)|}{|Z_i|}\right)^q\le \gamma^{18q}.$$
	Thus, by the linearity of expectation, we have that $\ex[X]\le{|Z_1|\choose 2}\cdot \gamma^{18q}\le n^2\gamma^{18q}$. On the other hand,
	$$\ex|S'|=\sum_{v\in Z_1}\pr[v\in S']=\sum_{v\in Z_1}\pr[Q_2\cup Q_3\subseteq N(v)]=\sum_{v\in Z_1}\prod_{i=2}^{3}\left(\frac{d(v,Z_i)}{|Z_i|}\right)^q\ge n\gamma^{2q}.$$
	So, there exists choices of $Q_2$ and $Q_3$ such that $\ex[|S'|-X]\ge n\gamma^{2q}-n^2\gamma^{18q}\ge \frac{1}{2}n^{2/3}$. Then the set $S$ obtained from deleting one vertex from every bad pair in $S'$ has the desired properties.
\end{proof}

\begin{proof}[Proof of Theorem~\ref{thm-RTK8}]
Fix constants $\delta, M', \epsilon$ as follows:
$$0<1/n_0\ll \delta<1/M' \ll\eps\ll \gamma\ll 1.$$
Assume that $G$ is an $n$-vertex graph with $n\geq n_0$ and $\alpha(G)\leq \delta \sqrt{n\log{n}}$.
Apply the regularity lemma (Lemma~\ref{lm:RL} with $k=1$) with $G, V(G),\emptyset, \eps, \eps^{-1}, 1$ and $M'$ playing the roles of $G,U_1,U_2, \eps,M, k$ and $M'$, respectively to obtain a regularity partition $V_1\cup\dots\cup V_m$ and the reduced graph $R=R(\eps,\gamma/2)$ of order $m$ with $\eps^{-1}\leq m \leq M'$. Note that $R$ is $K_4$-free by Lemma~\ref{lem-drc-j}. We say that a triangle $ijk$ in $R$ is \emph{chubby} if $d_{G}(V_{i'},V_{j'}) \geq 2/3+\gamma$ for some $i'j' \in {\{i,j,k\}\choose 2}$.
We first show that there is no chubby triangle in $R$.
\begin{claim}\label{cl-light-tri}
	No triangle in $R$ is chubby.
\end{claim}
\begin{proof}
	Suppose without loss of generality that $\{1,2,3\}$ induces a triangle in $R$ with $d(V_2,V_3)\ge 2/3+\gamma$. By the definition of regular pair, it is well-known that for each $i\in [3]$, there exists a subset $V_i^*\subseteq V_i$ such that $|V_i^*|=(1-2\eps)|V_i|$ and $\mincr(G[V_1^*,V_2^*,V_3^*])\ge \gamma|V_i^*|/3$. Applying Lemma~\ref{lem-drc-v} to $G[V_1^*,V_2^*,V_3^*]$ with $V_i^*$s playing the roles of $Z_i$s, we obtain a set $S\subseteq V_1^*$ of size at least $\frac{1}{3}\left(\frac{n}{m}\right)^{2/3}$ such that every pair $P\in \binom{S}{2}$ satisfies $|N(P,V_i^*)|\geq (\frac{\gamma}{3})^9n/m$ for each $i\in \{2,3\}$. As $\frac{1}{3}\left(\frac{n}{m}\right)^{2/3} \geq \alpha(G)$, the set $S$ contains an edge $uv\in E(G)$ with $|N(uv,V_i^*)|\geq (\frac{\gamma}{3})^9n/m$ for each $i\in \{2,3\}$.
 Using Lemma~\ref{lem-slicing} and again deleting low degree vertices, we get $V_i'\subseteq N(uv,V_i^*)$ for $i\in\{2,3\}$ such that $|V_2'|=|V_3'|\ge \gamma^{10}n/m$; $(V_2',V_3')$ is $(\sqrt{\eps}, \frac{2}{3}+\frac{\gamma}{4})$-regular with $\delta(G[V_2',V_3'])\ge (\frac{2}{3}+\frac{\gamma}{5})|V_3'|$. Observe that $V_2'$ must contain a triangle, as otherwise Theorem~\ref{thm: Shearer} implies that there exists an independent set of size at least
	\begin{equation}\label{eq-R3k}
		\frac{1}{2}\sqrt{|V_2'|\log|V_2'|}\ge\frac{1}{2}\sqrt{\gamma^{10}\frac{n}{m}\log\left(\gamma^{10}\frac{n}{m}\right)}> \frac{1}{m}\sqrt{n\log n}>\delta\sqrt{n\log n}\ge \alpha(G),
	\end{equation}
	a contradiction. Let $T$ be a triangle in $V_2'$. Then we have that 
	$$|N(T,V_3')|\ge 3\delta(G[V_2',V_3'])-2|V_3'|\ge \frac{\gamma}{2}|V_3'|\ge \gamma^{12} \frac{n}{m}.$$
	Almost identical calculation as~\eqref{eq-R3k} shows that $N(T,V_3')$ contains a triangle, which together with $u$, $v$ and $T$ forms a copy of $K_8$, a contradiction.
\end{proof}
We are now ready to prove Theorem~\ref{thm-RTK8}. 
Let $t\in \mathbb{R}$ be such that $e(R)=m^2/4+t$. If $t<0$, then \eqref{eq: size original reduced} with the definition of $R=R(\epsilon,\gamma/2)$ implies that 
$$e(G)\leq (m^2/4) (n/m)^2+ n^2/m + \gamma n^2/2 \leq n^2/4+ \gamma n^2$$ 
as $1/m\ll \gamma$. We may thus assume $t\geq 0$.

Recall that $R$ is $K_4$-free, Tur\'an's theorem implies that $e(R)\le m^2/3$ and $t/m^2\le 1/12$; and by Lemma~\ref{thm-disj-k3}, $E(R)$ can be decomposed into $m^2/4-2t$ edges and $t$ edge-disjoint triangles. Each triangle in $R$, by Claim~\ref{cl-light-tri}, corresponds to at most 
$$3\cdot (2/3+\gamma)\cdot (n/m)^2=(2+3\gamma)(n/m)^2$$
edges in $G$. Hence
$$e(G)\le \left(\frac{m^2}{4}-2t+(2+3\gamma)t\right)\frac{n^2}{m^2}+ \frac{n^2}{m}+ \frac{\gamma}{2} n^2\le \frac{n^2}{4}+\gamma n^2,$$
as desired.
\end{proof}
\section{Lower bound constructions for $\varrho(K_3,K_s,\delta)$}\label{sec-lowerbound-2col}
For each $s\in\{3,4,5\}$ and small $\delta >0$, we will construct an $n$-vertex $(K_3,K_s)$-free graph $G$ with $\alpha(G)\le \delta n$ and the desired edge-density. This provides a lower bound on $\varrho(K_3,K_{s},\delta)$.
Throughout this section, we use $X_1,\ldots,X_k$ for the partite sets of $T_k(n)$.

\subsection{Lower bound for $\varrho(K_3,K_s,\delta)$ when $s\in\{3,5\}$}
If $s=3$, let $G$ be a graph obtained from putting a copy of $F(\frac{n}{2},d)$, for some $d\in[\delta n-o(n),\delta n]$, in both partite sets of $T_2(n)$. It is easy to see that $\alpha(G)\le \delta n$  and $e(G)=\frac{n^2}{4}+\frac{\delta n^2}{2} + o(n^2)$. It is also easy to check that the following edge-colouring $\phi$ is a $(K_3,K_3)$-free colouring: $\phi(e)=1$ for all $e\in \cup_{i\in[2]}G[X_i]$; and $\phi(X_1,X_2)=2$, see Figure ~\ref{fig-lower}.

If $s=5$, let $G$ be a graph obtained from putting a copy of $F(\frac{n}{5},d)$, for some $d\in[\delta n-o(n),\delta n]$, in each partite set of $T_5(n)$. It is easy to see that $\alpha(G)\le \delta n$  and $e(G)=\frac{2n^2}{5}+\frac{\delta n^2}{2}+ o(n^2)$. It is also easy to check that the following edge-colouring $\phi$ is a $(K_3,K_5)$-free colouring: $\phi(e)=2$ for all $e\in \cup_{i\in[5]}G[X_i]$; $\phi(X_i,X_{i+1})=2$ for all $i\in[5]$ (addition modulo $5$); and all other edges are of colour 1, see Figure ~\ref{fig-lower}.

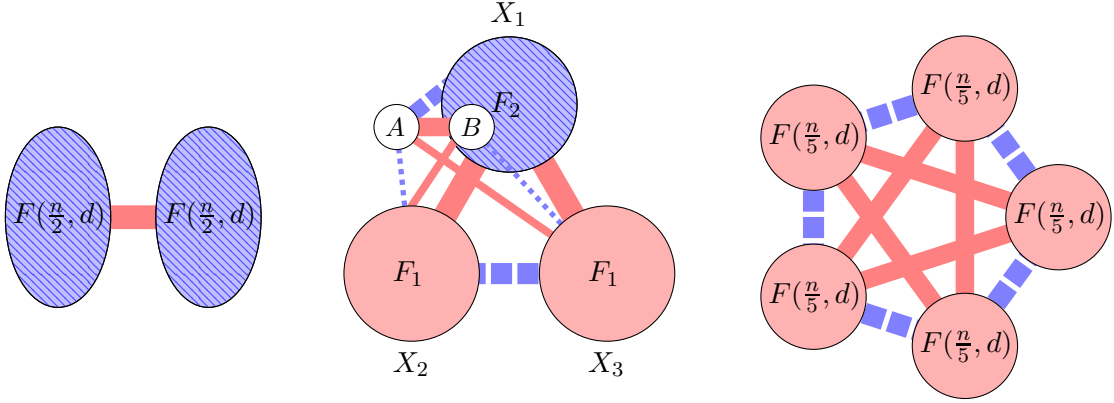
\begin{figure}
	\centering
	\begin{tikzpicture}
	
	\tikzset{be/.style = {line width=8pt, color=blue!50!white, dotted}}
	\tikzset{re/.style = {line width=9pt, color=red!50!white}}
	\tikzset{b2e/.style = {line width=6pt, color=blue!50!white, dotted}}
	\tikzset{r2e/.style = {line width=7pt, color=red!50!white}}
	\tikzset{b3e/.style = {line width=2pt, color=blue!50!white, dotted}}
	\tikzset{r3e/.style = {line width=3pt, color=red!50!white}}	

	\draw[re] (-3.3-2,0) to (-2.7-2,0);
	
				\draw[fill=blue!25]  (-4-2,0) ellipse (0.7cm and 1.2cm);
	\draw[pattern=north west lines, pattern color=blue!70]  (-4-2,0) ellipse (0.7cm and 1.2cm);
	\node [rectangle, fill=none, draw=none, text width = 1cm] at (-4-2-0.1,0) {$F(\frac{n}{2},d)$};
	
		\draw[fill=blue!25]  (-2-2,0) ellipse (0.7cm and 1.2cm);
	\draw[pattern=north west lines, pattern color=blue!70]  (-2-2,0) ellipse (0.7cm and 1.2cm);
	\node [rectangle, fill=none, draw=none, text width = 1cm] at (-2-2-0.1,0) {$F(\frac{n}{2},d)$};

		\draw[be] (1*1.8+5.5,0) to (0.309*1.8+5.5,0.951*1.8);
	\draw[be] (0.309*1.8+5.5,0.951*1.8) to (-0.8090*1.8+5.5,0.5878*1.8);
	\draw[be] (-0.8090*1.8+5.5,0.5878*1.8) to (-0.8090*1.8+5.5,-0.5878*1.8);
	\draw[be] (-0.8090*1.8+5.5,-0.5878*1.8) to (0.309*1.8+5.5,-0.951*1.8);
	\draw[be] (0.309*1.8+5.5,-0.951*1.8) to (1*1.8+5.5,0);
	
	\draw[r2e] (1*1.8+5.5,0) to  (-0.8090*1.8+5.5,0.5878*1.8);
	\draw[r2e] (0.309*1.8+5.5,0.951*1.8)  to (-0.8090*1.8+5.5,-0.5878*1.8);
	\draw[r2e] (-0.8090*1.8+5.5,0.5878*1.8) to (0.309*1.8+5.5,-0.951*1.8);
	\draw[r2e] (-0.8090*1.8+5.5,-0.5878*1.8) to (1*1.8+5.5,0);
	\draw[r2e] (0.309*1.8+5.5,-0.951*1.8) to (0.309*1.8+5.5,0.951*1.8);

	\draw[fill=red!30] (1*1.8+5.5,0) ellipse (0.7cm and 0.7cm);
	\node [rectangle, fill=none, draw=none, text width = 1cm] at (1*1.8-0.1+5.5,0) {$F(\frac{n}{5},d)$};
	
	\draw[fill=red!30] (0.309*1.8+5.5,0.951*1.8) ellipse (0.7cm and 0.7cm);
	\node [rectangle, fill=none, draw=none, text width = 1cm] at (0.309*1.8-0.1+5.5,0.951*1.8) {$F(\frac{n}{5},d)$};
	
	\draw[fill=red!30] (-0.8090*1.8+5.5,0.5878*1.8) ellipse (0.7cm and 0.7cm);
	\node [rectangle, fill=none, draw=none, text width = 1cm] at (-0.8090*1.8-0.1+5.5,0.5878*1.8) {$F(\frac{n}{5},d)$};
	
	\draw[fill=red!30] (-0.8090*1.8+5.5,-0.5878*1.8) ellipse (0.7cm and 0.7cm);
	\node [rectangle, fill=none, draw=none, text width = 1cm] at ((-0.8090*1.8-0.1+5.5,-0.5878*1.8) {$F(\frac{n}{5},d)$};

	\draw[fill=red!30] (0.309*1.8+5.5,-0.951*1.8) ellipse (0.7cm and 0.7cm);
	\node [rectangle, fill=none, draw=none, text width = 1cm] at (0.309*1.8-0.1+5.5,-0.951*1.8) {$F(\frac{n}{5},d)$};

	\draw[re] (0,1*1.5,-0) to (-0.8660*1.5,-0.5*1.5);
	\draw[be] (-0.8660*1.5,-0.5*1.5) to (0.8660*1.5,-0.5*1.5);
	\draw[re] (0.8660*1.5,-0.5*1.5) to (0,1*1.5);

\draw[b2e] (-1.5,1.2) to (-0.5,2);

	\draw[fill=blue!25] (0,1*1.5) ellipse (0.9cm and 0.9cm);
	\draw[pattern=north west lines, pattern color=blue!70] (0,1*1.5) ellipse (0.9cm and 0.9cm);
	\node [rectangle, fill=none, draw=none, text width = 1cm] at (0+0.25,1*1.5) {$F_2$};
		\node [rectangle, fill=none, draw=none, text width = 1cm] at (0+0.25,1*1.5+1.2) {$X_1$};
		
	\draw[r2e] (-1.5,1.2) to (-0.5,1.2);

	\draw[b3e] (-1.5,1.2) to (-0.8660*1.5,-0.5*1.5);
\draw[r3e] (-1.5,1.2) to (0.8660*1.5,-0.5*1.5);

\draw[r3e] (-0.6,1.2) to (-0.8660*1.5-0.6,-0.5*1.5);
\draw[b3e] (-0.5,1.2) to (0.8660*1.5,-0.5*1.5);

	\draw[fill=white] (0-0.5,1*1.5-0.3) ellipse (0.3cm and 0.3cm);
	\node [rectangle, fill=none, draw=none, text width = 1cm] at (0-0.5+0.35,1*1.5-0.3) {{\small $B$}};

	\draw[fill=white] (0-1.5,1*1.5-0.3) ellipse (0.3cm and 0.3cm);
	\node [rectangle, fill=none, draw=none, text width = 1cm] at (0-1.5+0.35,1*1.5-0.3) {{\small $A$}};	
	
	\draw[fill=red!30] (-0.8660*1.5,-0.5*1.5) ellipse (0.9cm and 0.9cm);
	\node [rectangle, fill=none, draw=none, text width = 1cm] at (-0.8660*1.5+0.25,-0.5*1.5){$F_1$};
	\node [rectangle, fill=none, draw=none, text width = 1cm] at (-0.8660*1.5+0.25,-0.5*1.5-1.2){$X_2$};
	
	\draw[fill=red!30] (0.8660*1.5,-0.5*1.5) ellipse (0.9cm and 0.9cm);
	\node [rectangle, fill=none, draw=none, text width = 1cm] at  (0.8660*1.5+0.25,-0.5*1.5)  {$F_1$};
		\node [rectangle, fill=none, draw=none, text width = 1cm] at  (0.8660*1.5+0.25,-0.5*1.5-1.2)  {$X_3$};
	
	\end{tikzpicture}
	\caption{Constructions for $s=3,4,5$ in order. Colour $1$: dotted blue. Colour $2$: red.}\label{fig-lower}
\end{figure}

\subsection{Lower bound for $\varrho(K_3,K_4,\delta)$.}
 We construct an $n$-vertex $(K_3,K_4)$-free graph $G$ with $\alpha(G)\le \delta n$ and $(\frac{1}{3}+\frac{\delta}{2}+\frac{3\delta^2}{2}-o(1))n^2$ edges as follows.
\begin{itemize}
	\item By Theorem~\ref{thm-Brandt}, there exist $F_1:=F(\frac{n}{3}, d_1)$ and $F_2:=F(\frac{n}{3}-\delta n, d_2)$ where $d_i\in [\delta n- o(n),\delta n]$ for each $i\in [2]$. So $e(F_1)=\frac{\delta n^2}{6}+ o(n^2)$ and $e(F_2)=\frac{\delta n^2}{6}-\frac{\delta^2n^2}{2}+ o(n^2)$.
	
	\item Let $B=\{b_1,b_2,\ldots, b_{d_2}\}$ be an independent set of size $d_2$ in $F_2$. Let $F$ be an $n/3$-vertex graph obtained from $F_2$ by 
	\begin{itemize}
		\item first adding a clone set of $B$, i.e.~a set of $d_2$ new vertices $A=\{a_1,a_2,\ldots,a_{d_2}\}$ with $N_F(a_i):=N_{F_2}(b_i)$ for each $i\in [d_2]$; 
		\item adding all $[A,B]$-edges; and 
		\item adding an additional set of $\delta n-d_2$ isolated vertices.
	\end{itemize}
	Note that $F$ is \emph{not} triangle-free, and
	$$e(F)=e(F_2)+d_2^2+d_2^2=\frac{\delta n^2}{6}+\frac{3\delta^2n^2}{2}+ o(n^2).$$
	\item Finally, let $G$ be the graph obtained from $T_3(n)$ on partite sets $X_i$, $i\in[3]$, by putting a copy of $F$ in $X_1$ and a copy of $F_1$ in $X_2$ and $X_3$. 
\end{itemize}
It is clear that $G$ has the desired size and easy to check that the following 2-edge-colouring $\phi$ of $G$ is $(K_3,K_4)$-free, see Figure ~\ref{fig-lower}: 
\begin{itemize}
	\item let $\phi(A,X_2)=\phi(B,X_3)=\phi(X_2,X_3)=1$; 
	\item let $\phi(e)=1$, for all $e\in E(G[X_1]\setminus [A,B])$; 
	\item all other edges are of colour 2. 
\end{itemize}

\section{Upper bound for $\varrho(K_3,K_4,\delta)$}\label{sec-K3Ks}
For the convenience of the reader, we rephrase the upper bound as follows.

\begin{lemma}\label{thm-K3K4-upperbound}
Suppose $0<1/n \ll \delta \ll 1$.
Let $G$ be an $n$-vertex $(K_3,K_4)$-free graph with $\alpha(G)\le \delta n$. Then $$e(G)\le \frac{n^2}{3}+\frac{\delta n^2}{2}+\frac{3\delta^2n^2}{2}.$$ 
\end{lemma}

We use the stability approach. A weak stability was proven in~\cite{EHSSSz}\footnote{Their proof missed a case, which can be easily fixed. We include it in the online arXiv version.}, stating that an $n$-vertex $(K_3,K_4)$-free graph $G$ with $\alpha(G)\leq \delta n$ is close to $T_3(n)$. For the exact result for $\varrho(K_3,K_4,\delta)$, we need a coloured stability, which roughly says that any $(K_3,K_4)$-free colouring of an almost extremal graph should look similar to the colouring given in the lower bound construction.

\subsection{Coloured stability} 
\begin{lemma}\label{thm-K3K4-stab}
Suppose $0< 1/n \ll \delta\ll\gamma<1$. Let $G$ be an $n$-vertex $(K_3,K_4)$-free graph with $\alpha(G)\le \delta n$ and $\delta(G)\ge 2n/3$. Then for any $(K_3,K_4)$-free $2$-edge colouring $\phi: E(G)\rightarrow [2]$, there exists a partition $X_1\cup X_2\cup X_3$ of $V(G)$ such that the following holds.
	\begin{enumerate}[label={\rm (A\arabic*)}]
		\item\label{A1} For each $i\in [3]$, we have $|X_i|=n/3\pm 2\gamma^{1/7} n$.
		\item\label{A2} $\alpha(G_1[X_1])\le \gamma^{1/4} n$.
		\item\label{A3} For each $i\in [3]$, we have $\Delta(G[X_i])\le \gamma^{1/18}n$.
		\item\label{A4} For each $v\in X_1$, we have $\min\{d_{G_1}(v,X_2),~d_{G_1}(v,X_3)\}<\gamma^{1/9}n$.
		\item\label{A5} $\mincr(G[X_1,X_2,X_3])\ge n/3-\gamma^{1/19} n.$
		\item\label{A6} For all $i\in\{2,3\}$ and $v\in X_i$, we have $d_{G_2}(v,X_1)\ge |X_1|-\gamma^{1/20}n$.
	\end{enumerate}
	Furthermore, one of the following occurs:
	\begin{enumerate}[label={\rm (P\arabic*)}]
		\item\label{P1} For all $i\in \{2,3\}$ and $v\in X_i$, we have $\alpha(G_2[X_i])\le \gamma^{1/4} n$ and $d_{G_1}(v,X_{5-i})\ge |X_{5-i}|-\gamma^{1/20}n$.
		\item\label{P2} For all $i\in \{2,3\}$ and $v\in X_i$, we have
		$\alpha(G_1[X_i])\le \gamma^{1/4} n$ and $d_{G_2}(v,X_{5-i})\ge |X_{5-i}|-\gamma^{1/20}n$.
	\end{enumerate}
\end{lemma}

We need an additional definition for the proof of Lemma~\ref{thm-K3K4-stab}.
\begin{definition}
	Given a weighted graph $R$ with weight $d: E(R)\rightarrow (0,1]$ and $Y\subseteq X\subseteq V(R)$, a \emph{$\gamma$-generalised clique} of order $t:=|X|+|Y|$, denoted by $Z_t$, on $(X,Y)$ is a clique on $X$ with $d(e)>1/2+\gamma$ for every $e\in E(R[Y])$.
\end{definition}

For brevity, we will write $(a_1a_2\ldots a_s,b_1b_2\ldots b_{s'})$ for $(\{a_1,\ldots,a_s\},\{b_1,\ldots,b_{s'}\})$. 
		
Note that, for an $n$-vertex $2$-edge-coloured graph $G$ with $\alpha(G)=o(n)$, both $G_1, G_2$ can have $\Omega(n)$ independence number. We will use the following lemma combined with regularity lemma to obtain a regular partition such that each part of the partition induces a graph with small independence number in one of the two colours.
\begin{lemma}[Lemma 3.1 in~\cite{BLSh} with $r=2$]\label{lem-split-2-col}
	Let $c>0$, $G$ be an $n$-vertex graph with $\alpha(G)\le c^2n$ and $\phi: E(G)\rightarrow [2]$. Then there exists a partition $V(G)=V_1^*\cup V_2^*$\footnote{It could be that some $V_i^*$ is empty.} such that for every $i\in [2]$, $\alpha(G_i[V_i^*])\le cn$.
\end{lemma}

\begin{proof}[Proof of Lemma~\ref{thm-K3K4-stab}]
We choose the constants as follows:
$$0< 1/n\ll\delta\ll \delta^* \ll 1/m \ll\eps\ll \gamma\ll 1.$$
We apply Lemma~\ref{lem-split-2-col} with $c=\delta^{1/2}$ to obtain a partition $V_1^*\cup V_2^*$ such that $\alpha(G_i[V_i^*])\leq \delta^{1/2} n$.
 Apply Theorem~\ref{lm:RL} with $G, V^*_1, V^*_2,\phi, \eps, \epsilon^{-1}$ and $M'$ playing the roles of $G,U_1,U_2,\phi, \epsilon,M$ and $M'$ to obtain an $\eps $-regular partition $V_1\cup \dots\cup V_m$ with $\epsilon^{-1}\leq m\leq M'$ which refines the partition $V_1^*\cup V_2^*$. Let $R:=R(\eps ,\gamma,\phi,(V_i)_{i\in[m]})$ be its reduced graph. It was shown in~\cite{EHSSSz} (Theorem~3(b) and~(e)) that $|G\triangle T_3(n)|\le \delta^* n^2$. As a consequence, the number of $K_4$ in $G$ is at most $\delta^* n^4$. 
 It is well-known that the reduced graph $R$ essentially inherits the structure of $G$: $\delta(R)\ge (2/3-3\gamma)m$ and $R$ is $K_4$-free. Indeed, if $K_4\subseteq R$, then by Lemma~\ref{lem-embedding}, $G$ contains at least $(\gamma/2)^6(n/m)^4/2 > \delta  n^4$ copies of $K_4$, a contradiction. Thus, by Theorem~\ref{thm-furedi-stab}, 
\begin{equation}\label{eq-R-close-to-T3}
	|R\triangle T_3(m)|\le \gamma^{1/3}  m^2.
\end{equation}

We define a colouring $\phiind:V(R)\cup E(R)\rightarrow [2]$, induced by $\phi$, as follows: 
\begin{itemize}
	\item[(i)] for each $k\in [m]$, we have $\phiind(k)=i$ if $V_k\subseteq V_i^*$; and
	
	\item[(ii)] for each $pq\in E(R)$, we have $\phiind(pq)=1$ if $d_{G_1}(V_p,V_q)\ge\gamma$, and $\phiind(pq)=2$ if $d_{G_1}(V_p,V_q)<\gamma$ and $d_{G_2}(V_p,V_q)\ge\gamma$.
\end{itemize}
We remark that colour $1$ has ``higher priority'' on $E(R)$ in $\phiind$, i.e.~if $(V_i,V_j)$ is dense in both $G_1$ and $G_2$, then we have $\phiind(ij)=1$. This asymmetry is needed for the embedding later. 
For each $pq$, we let $d(pq):= d_{G_{\phiind(pq)}}(V_p, V_q)$ be the weight on $E(R)$, and we consider $R$ as a weighted graph.
 It is also well-known that for each $p\in V(R)$, we have
 \begin{align}
 \sum_{q\in N_{R}(p)} d(pq) \geq \left(\frac{m}{n}\right)^2\left(\delta(G)\frac{n}{m} - \frac{\epsilon n^2}{m} - \frac{2\gamma n^2}{m} - \left(\frac{n}{m}\right)^2 \right) \geq (2/3- 3\gamma) m.
 \end{align}

Let $R'$ be the graph obtained from $R$ by deleting all edges of weight at most $1/2+\gamma$. Then for each $p\in V(R)$, we have
$$(2/3- 3\gamma) m  \leq \sum_{q\in N_{R}(p)} d(pq) \leq d_{R'}(p) + (1/2+\gamma) (d_{R}(p)-d_{R'}(p)),$$ 
thus, as $\delta(R)\ge(2/3-3\gamma)m$, we have
\begin{align}\label{eq: R' min deg}
d_{R'}(p) \geq 4m/3 - d_{R}(p) - 9\gamma m.
\end{align}
Moreover, by \eqref{eq-R-close-to-T3}, we know $e(R)\le \frac{m^2}{3}+\gamma^{1/3}m^2$. As $\delta(G)\geq 2n/3$, similar to~\eqref{eq: size original reduced}, we have 
$$n^2/3\le e(G)\le (e(R)-e(R'))\frac{(1/2+\gamma) n^2}{m^2}+e(R')\frac{n^2}{m^2}+(2\gamma+\epsilon + 1/m) n^2.$$
This implies
\begin{align}\label{eq: R' deleted}
e(R)-e(R') \leq \gamma^{1/4}m^2.
\end{align}

We will omit $\gamma$ in the term `$\gamma$-generalised clique'.
For each $i\in [2]$ and $Y\subseteq X\subseteq V(R)$, we say that a generalised clique $Z_t$ in $R$ on $(X,Y)$ is \emph{of colour $i$} if $\phiind(k)=\phiind(pq)=i$, for all $k\in Y$ and $pq\in {X\choose 2}$. We say that $R$ is \emph{$(Z_{t_1},Z_{t_2})$-free} if there is no $Z_{t_i}$ of colour $i$ for any $i\in[2]$. It was implicitly proven in the proof Theorem~1.3 in~\cite{BLSh} that a $Z_t$ of colour $i$ in $R$ implies $K_t\subseteq G_i$. This implies the following, since $G$ is $(K_3,K_4)$-free:
\begin{equation}\label{eq-generlised-clique}
R~ \mbox{ is } ~(Z_3,Z_4)\mbox{-free}.
\end{equation}

Let $U_1\cup U_2\cup U_3$ be a max-cut $3$-partition of $R$. The desired partition of $V(G)$ will be an adjustment of this partition.
By \eqref{eq-R-close-to-T3} and the definition of max-cut and Theorem~\ref{thm-furedi-stab}, it is easy to see that we have
	\begin{equation}\label{eq-Ui-property}
	\sum_{i\in[3]}e(R[U_i])\le \gamma^{1/3} m^2, \quad \quad |U_i|=m/3\pm \gamma^{1/7} m  \enspace \text{and}
	\end{equation}
	\begin{equation}\label{eq-Ui-mincr}
	\mincr(R[U_1,U_2,U_3])\ge (\delta(R)-\max_{i\in[3]}|U_i|)/2\ge m/7.
	\end{equation}


	We will obtain the colour pattern of $R$ in $\phiind$. First we show that each vertex set $U_i$ is monochromatic in $\phiind$.
	\begin{claim}\label{cl-R-Vi-mono}
		For every $i\in[3]$, there exist $j\in [2]$ such that $\phiind(U_i)=j$. In particular, we have
		$$\alpha(G_j[\cup_{k\in U_i}V_k])\le \sqrt{\delta}n.$$	
	\end{claim}
	\begin{proof}
Suppose the lemma is not true, then by symmetry, we may assume that $\phi(U_1)\neq j$ for any $j\in [2]$. Let $W:=\{ w \in U_1: \phiind(w)=2\}$.
We shall argue that one of the following two cases must happen and then derive contradictions in each case.

\noindent {\bf Case 1.} There exists vertices $u,w \in U_1$
 $v_2\in U_2$, $v_3\in U_3$ such that $\{v_2v_3, uv_2,uv_3\} \subseteq E(R)$  and $wv_2,wv_3 \subseteq E(R')$ and $ \phiind(u)=1, \phiind(w)=2$.
 
\noindent {\bf Case 2.}  There exists vertices $u,w \in U_1$
$v_2\in U_2$, $v_3\in U_3$ such that $R[\{u,w,v_2,v_3\}]$ induces a copy of $K_4$ and $\phiind(u)=1, \phiind(w)=2$.

Suppose that $|W|\geq m/100$. Fix an arbitrary $u\in U_1$ with $\phiind(u)=1$. Then, by \eqref{eq: R' deleted}, more than half of the vertices $w$ in $W$ satisfy $d_{R'}(w)\geq d_{R}(w)-\gamma^{1/5}m$, and as $\sum_{i\in[3]}e(R[U_i])\le \gamma^{1/3} m^2$, more than half of the vertices $w$ in $W$ satisfies $|N_{R}(w,U_1)| \leq \gamma^{1/4}m$.
Hence there exists $w\in W$ with $|N_{R'}(w, U_i)|\ge \delta(R) - |U_{5-i}| - 2\gamma^{1/5}m \geq m/4$ for each $i\in\{2,3\}$. 
By this and \eqref{eq-Ui-mincr}, for each $i\in \{2,3\}$ we have $$|N_R(u,U_i)\cap N_{R'}(w,U_i)|\ge m/7+m/4-|U_i|\ge m/30.$$ 
Together with \eqref{eq-R-close-to-T3} and the definition of max-cut partition, this
implies that there exists an edge $v_2v_3$ between $N_R(u,U_2)\cap N_{R'}(w,U_2)$ and $N_R(u,U_3)\cap N_{R'}(w,U_3)$, yielding Case 1.

We may then assume that  $|W|\leq m/100$. Fix an arbitrary $w\in W$. If $|N_{R}(w,U_1)|> m/50$, then we have $ |N_{R}(w,U_1\setminus W)|\geq m/100$.
As $\sum_{i\in[3]}e(R[U_i])\le \gamma^{1/3} m^2$, more than half of the vertices $u$ in $N_{R}(w,U_1\setminus W)$ satisfy $|N_{R}(u,U_1)| \leq \gamma^{1/4}m$.
Hence there exists $u \in N_{R}(w,U_1\setminus W)$ with $|N_{R}(u, U_i)|\geq m/4$ for each $i\in\{2,3\}$. 
By this and \eqref{eq-Ui-mincr}, for each $i\in \{2,3\}$ we have $$|N_R(u,U_i)\cap N_{R}(w,U_i)|\ge m/7+m/4-|U_i|\ge m/30.$$ 
Thus by \eqref{eq-R-close-to-T3} and the definition of max-cut partition, there exists an edge $v_2v_3$ between $N_R(uw,U_2)$ and $N_R(uw,U_3)$, yielding Case 2. 

Thus we may assume that $|N_{R}(w,U_1)|\leq m/50$, thus $d_{R}(w) \leq |U_2|+|U_3| + m/50\leq (2/3+ 1/40)m$. 
Together with \eqref{eq: R' min deg}, this implies that 
$$d_{R'}(w)\geq 4m/3 - (2/3+ 1/40)m - 9\gamma m \geq (2/3 -1/30)m.$$
Hence, for each $i\in \{2,3\}$, we have
$$|N_{R'}(w,U_i)| \geq d_{R'}(w) - |N_{R}(w,U_1)| - |U_{5-i}| 
\geq (1/3-1/30-1/40)m \geq m/4.$$
By  \eqref{eq-Ui-mincr}, there exists a vertex $u\in U_1\setminus W$ such that 
for each $i\in \{2,3\}$, we have \begin{align*}
|N_{R}(u,U_{i})\cap N_{R'}(w,U_{i})| &\geq m/4 + m/7 - |U_{2}| \geq m/30.
\end{align*}
Thus by \eqref{eq-R-close-to-T3} and the definition of max-cut partition, there exists an edge $v_2v_3$ between $N_{R}(u,U_{2})\cap N_{R'}(w,U_{2})$ and $N_{R}(u,U_{3})\cap N_{R'}(w,U_{3})$, yielding again Case 1. 

For each $i\in \{2,3\}$ and $j\in [2]$, if $\phiind(U_i)=j$, then, by the definition of $\phiind$, we have $\cup_{k\in U_i}V_k\subseteq V_j^*$, and so $\alpha(G_j[\cup_{k\in U_i}V_k])\le \alpha(G_j[V_j^*])\le \sqrt{\delta}n$ as desired. We shall now derive contradictions in each case to finish the proof.

Suppose {\bf Case 1} happens. By the definition of $R'$, for each $i\in \{2,3\}$ we have $d(wv_i)\geq 1/2+\gamma$.
As $\phiind(u)=1$, we must have $\phiind(uv_i)=2$ for $i\in\{2,3\}$, otherwise we get a $Z_3$ of colour $1$ on $(uv_i,u)$, contradicting~\eqref{eq-generlised-clique}. 
		Suppose now that $\phiind(v_2v_3)=2$. For each $i\in\{2,3\}$, it must be that $\phiind(v_i)=1$, otherwise $(uv_2v_3,v_i)$ is a $Z_4$ of colour $2$, which in turn implies that $\phiind(wv_i)=2$, otherwise $(wv_i,v_i)$ is a $Z_3$ of colour $1$. But then $(wv_2v_3,w)$ is a $Z_4$ of colour 2, a contradiction.
		
		Hence, we may assume that $\phiind(v_2v_3)=1$. For each $i\in\{2,3\}$, we must have $\phiind(v_i)=2$, otherwise we get a $Z_3$ of colour $1$ on $(v_2v_3,v_i)$, a contradiction. As $d(wv_i)\ge 1/2+\gamma$ and $\phiind(w)=2$, we must have $\phiind(wv_i)=1$, otherwise we get a $Z_4$ of colour $2$ on $(wv_i,wv_i)$. However, then we have a $Z_3$ of colour $1$ on $(wv_2v_3,\emptyset)$, a contradiction.

Suppose {\bf Case 2} happens. As $\phiind(u)=1$, we must have $\phiind(uw)= \phiind(uv_i)=2$ for $i\in\{2,3\}$, otherwise we get a $Z_3$ of colour $1$ on $(uv_i,u)$ or $(uw,u)$, contradicting~\eqref{eq-generlised-clique}. 
		
		Suppose now that $\phiind(v_2v_3)=2$. Then for each $i\in \{2,3\}$, we have $\phiind(v_i)=1$, otherwise $(uv_2v_3,v_i)$ is a $Z_4$ of colour $2$, which in turn implies that $\phiind(wv_i)=2$ for each $i\in \{2,3\}$. But then $(wuv_2v_3,\emptyset)$ is a $Z_4$ of colour $2$, a contradiction.
		
		Hence, we may assume that $\phiind(v_2v_3)=1$, Then for each $i\in\{2,3\}$, we must have $\phiind(v_i)=2$, otherwise we get $Z_3$ of colour $1$ on $(v_2v_3,v_i)$. Moreover, for each $i\in \{2,3\}$, we must have $\phiind(wv_i)=1$, otherwise $(v_iuw, w)$ is a $Z_4$ of colour $2$. But then $(wv_2v_2,\emptyset)$ forms a $Z_3$ of colour $1$, a contradiction.
\end{proof}
	
	\begin{claim}\label{cl-R-cp}
By permuting indices of $U_1,U_2,U_3$, we may assume the following.
We have $\phiind(U_1)=1$ and for each $i\in \{2,3\}$, we have $\phiind(U_1,U_i)=2$ and one of the following holds.
\begin{enumerate}[label={\rm (B\arabic*)}]
\item \label{B1} $\phiind(U_2)=\phiind(U_3)=2$ and $\phiind(U_2,U_3)=1$; or
\item \label{B2} $\phiind(U_2)=\phiind(U_3)=1$ and $\phiind(U_2,U_3)=2$.
\end{enumerate}
\end{claim}
	\begin{proof}
	If $\phiind(U_i)=2$ for all $i\in [3]$, then it is easy to see that all crossing edges of $R'$ are of colour $1$, otherwise we obtain a generalised clique $Z_4$ of colour $2$. However, then we can easily check that $R$ contains a copy of $K_3$ of colour $1$, which is again a contradiction.

 Hence, by Claim~\ref{cl-R-Vi-mono}, we may assume that $\phiind(U_1)=1$. Then as $R$ does not have a generalised clique $Z_3$ of colour $1$,  we have that $\phiind(U_1,U_i)=2$ for $i\in\{2,3\}$. If $\phiind(U_2)=2$, then $\phiind(U_2,U_3)=1$, otherwise we get a generalised clique $Z_4$ of colour $2$. But then we must have $\phiind(U_3)=2$, giving \ref{B1}. Similarly if $\phiind(U_2)=1$, we obtain \ref{B2}.
	\end{proof}	
	Let $V(G)=X_1'\cup X_2'\cup X_3'$ be the partition corresponding to $V(R)=U_1\cup U_2\cup U_3$, i.e.~$X_i':=\cup_{k\in U_i}V_k$, then \eqref{eq-Ui-property} implies that for each $i\in [3]$ we have $|X_i'|=\frac{n}{3}\pm \frac{3}{2} \gamma^{1/7} n$.
	Note that we have
		\begin{equation}\label{eq-eXi-prime}
			\sum_{i\in[3]}e(G[X_i'])\le \sum_{i\in[3]}e(R[U_i])\cdot \left(\frac{n}{m}\right)^2+\epsilon n^2 + \frac{n^2}{m} + 2\gamma n^2\stackrel{(\ref{eq-Ui-property})}{\le} 2\gamma^{1/3} n^2.
		\end{equation} 
	  Note that \eqref{eq-Ui-mincr} provides a minimum crossing degree of $R$ with respect to the partition $U_1\cup U_2\cup U_3$. However, in $G$, some vertex could have low crossing degree with respect to the partition $X_1'\cup X'_2 \cup X'_3$. 
	To amend this problem, we will consider the following modified partition $X_1\cup X_2\cup X_3$ of $V(G)$.

	\begin{claim}\label{cl-refine-Vi}
	There exists a partition $X_1\cup X_2\cup X_3$ of $V(G)$ such that the following holds.
	\begin{enumerate}[label={\rm (X\arabic*)}]
	\item\label{X1} For each $i\in [3]$, we have $|X_i|= n/3 \pm 2\gamma^{1/7}n$ and $|X_i\triangle X'_i| \leq 20\gamma^{1/3} n$.
	\item\label{X2} $\mincr(G[X_1,X_2,X_3])\geq n/10$.
	\end{enumerate}
	\end{claim}
	\begin{proof}
For all $i\in[3]$ and $v\in X_i'$, if $d(v, X_j')\le n/10$ for some $j\neq i$, then move $v$ to $X_j'$. We repeat this until no such vertex exists. Let the resulting set be $X_i$, $i\in [3]$. We first show that this process terminates and so $X_i$s are well-defined.	

Recall that $\delta(G)\ge 2n/3$, so if there exist $ij\in{[3]\choose 2}$ and $v\in X_i'$ with $d(v, X_j')\le n/10$, we see that for each $k\neq j$, 
    $$d(v,X_k')\ge \delta(G)-n/10-\max_{i\in[3]}|X_i'|\ge n/5.$$ Thus, after moving $v$ from $X_i'$ to $X_j'$, the number of inner edges decreases by at least $n/5-n/10=n/10$. Hence, by \eqref{eq-eXi-prime}, after moving at most $2\gamma^{1/3}n^2/(n/10)= 20\gamma^{1/3} n$ vertices, the process stops.
Hence, we obtain \ref{X1} proving the first part and \ref{X2} holds by definition.
	\end{proof}
Note that \ref{A1} holds due to \ref{X1}.
By Claims~\ref{cl-R-Vi-mono},~\ref{cl-R-cp} and~\ref{X1}, we have \ref{A2} as
	\begin{equation}\label{eq-ind-X1}
	\alpha(G_1[X_1])\le \alpha(G_1[X_1'])+||X_1'|-|X_1||\le \sqrt{\delta}n+20\gamma^{1/3} n\le \gamma^{1/4} n.
	\end{equation}
	For what follows, we assume \ref{B1} holds, which then leads to \ref{P1} (\ref{B2} implying \ref{P2} can be proven analogously). Similar to~\eqref{eq-ind-X1}, \ref{B1} implies that $\alpha(G_2[X_i])\le \gamma^{1/4} n$ for $i\in \{2,3\}$, proving the first part of~\ref{P1}.
	
		We now bound $\Delta(G[X_i])$ for each $i\in \{2,3\}$. Without loss of generality, it is enough to bound $\Delta(G[X_2])$. Note first that, as $G_1$ is $K_3$-free, by~\eqref{eq-ind-X1}, for each $v\in V(G)$, we have
	\begin{equation}\label{eq-all2-X1}
d_{G_1}(v,X_1)\le \alpha(G_1[X_1])\le \gamma^{1/4} n.
	\end{equation}
	Define 
	$$J:=\bigcup_{i\in [3]}\{v\in X_i:~ d(v,\overline{X_i})\le |\overline{X_i}|-\gamma^{1/8}n\}$$ 
	to be the set of vertices with large missing crossing degree. By Claim~\ref{cl-refine-Vi} and~\eqref{eq-eXi-prime}, we have $$\sum_{i\in[3]}e(G[X_i])\le \sum_{i\in[3]}(e(G[X_i'])+||X_i'|-|X_i||n)\le \gamma^{1/4} n^2,$$
	and so, as $e(G)\ge \frac{n^2}{3}$ and $e(K_{|X_1|,|X_2|,|X_3|})\leq n^2/3$, we have that
	\begin{eqnarray*}
		e(\overline{G}[X_1,X_2,X_3])&\le& \frac{n^2}{3}-(e(G)-\sum_{i\in[3]}e(G[X_i]))\le \sum_{i\in[3]}e(G[X_i])\le \gamma^{1/4} n^2 \enspace \text{and}
	\end{eqnarray*}
\begin{align}\label{eq: J size}
|J|\le \frac{2e(\overline{G}[X_1,X_2,X_3])}{\gamma^{1/8}n}=2\gamma^{1/8}n.
\end{align}
We claim that for each $y\in X_3$, we have
	\begin{equation}\label{eq-X3-vx-pattern}
	d_{G_2}(y,X_2)\le  3\gamma^{1/8}n ~\mbox{ and } ~d_{G_1}(y,X_3)\le |J|\le 2\gamma^{1/8}n.
	\end{equation}
	Indeed, suppose that $d_{G_2}(y,X_2)>3\gamma^{1/8}n$. Then $|N_{G_2}(y,X_2)\setminus J|\ge \gamma^{1/8}n>\alpha(G_2[X_2])$, and so there exists $uv\in E(G_2)$ with $u,v\in N_{G_2}(y,X_2)\setminus J$. By \ref{X2},~\eqref{eq-all2-X1} and the definition of $J$, we have 
	$$|N_{G_2}(\{u,v,y\}, X_1)|\ge \mincr(G[X_1,X_2,X_3])-\sum_{x\in \{u,v,y\}}d_{G_1}(x,X_1)-2\cdot \gamma^{1/8}n\ge n/20,$$
	showing that $K_4\subseteq G_2$, a contradiction, thus the first part of \eqref{eq-X3-vx-pattern} holds.
	
	Suppose that $d_{G_1}(y,X_3)>|J|$. So there exists $yy'\in E(G_1[X_3])$ with $y'\notin J$. But then the first part of~\eqref{eq-X3-vx-pattern} implies that $$|N_{G_1}(\{y,y'\}, X_2)|\ge \mincr(G[X_1,X_2,X_3])-\sum_{x\in \{y,y'\}}d_{G_2}(x,X_2)-\gamma^{1/8}n\ge n/20,$$
	contradicting $K_3\not\subseteq G_1$. Thus \eqref{eq-X3-vx-pattern} holds.
	
	We now show that for each $y\in X_3$, we have $d_{G_2}(y,X_3)\le 3\gamma^{1/17}n$, which together with~\eqref{eq-X3-vx-pattern} implies that $\Delta(G[X_3])\le \gamma^{1/18}n$. Fix an arbitrary $y\in X_3$ and let $Y:=N_{G_2}(y,X_3)$. suppose to the contrary that $|Y|>3\gamma^{1/17}n$. For $i\in[2]$, define
	$$J_i:=\{v\in X_i:~ d_{G_i}(v,X_3)\ge \gamma^{1/16}n\}.$$
	By~\eqref{eq-all2-X1} and~\eqref{eq-X3-vx-pattern}, we get, for each $i\in [2]$, that 
\begin{align}\label{eq: Ji size}
|J_i|\le \frac{e(G_i[X_i,X_3])}{\gamma^{1/16}n}\le \frac{|X_3|\cdot\max\{\gamma^{1/4}n,3\gamma^{1/8} n\}}{\gamma^{1/16}n}\le 3\gamma^{1/16}n.
\end{align}
	As $|N_{G_2}(y,X_1)|\ge \mincr(G[X_1,X_2,X_3])-d_{G_1}(y,X_1)>|J|+|J_1|$ due to~\eqref{eq-all2-X1} and \eqref{eq: J size}, we can pick $u\in N_{G_2}(y,X_1)\setminus (J\cup J_1)$. By the definition of $J$ and $J_1$, we have
	\begin{equation*}
	d_{G_2}(u,Y)\ge |Y|-d_{\overline{G}}(u,\overline{X_1})-d_{G_1}(u,X_3)\ge |Y|-\gamma^{1/17}n.
	\end{equation*}
	Similarly, we can pick $v\in X_2\setminus (J\cup J_2)$ with $d_{G_1}(v,Y)\ge |Y|-\gamma^{1/17}n$. Thus, writing $Y':=N_{G_1}(v,Y)\cap N_{G_2}(u,Y)$, we have $|Y'|\ge |Y|-2\gamma^{1/17}n> \alpha(G)$. So there exists $xx'\in E(G[Y'])$. However, if $\phi(xx')=1$, then $\{x,x',v\}$ induces a $K_3$ in $G_1$; while if $\phi(xx')=2$, then $\{x,x',u,y\}$ induces a $K_4$ in $G_2$, a contradiction. This shows $\Delta(G[X_3])\leq \gamma^{1/18}n$, and $\Delta(G[X_2])\leq \gamma^{1/18}n$.

	To bound $\Delta(G[X_1])$, we need to first prove \ref{A4} that no vertex in $X_1$ can have high $G_1$-degree to both $X_2$ and $X_3$. Suppose that $v\in X_1$ is such that $d_{G_1}(v,X_i)\ge \gamma^{1/9}n>|J|$ for both $i\in\{2,3\}$. Fix an arbitrary $u\in N_{G_1}(v,X_3)\setminus J$. Then by~\eqref{eq-X3-vx-pattern} and the fact that $u\notin J$, we have
	$$|N_{G_1}(\{v,u\}, X_2)|\ge d_{G_1}(v,X_2)-d_{\overline{G}}(u,\overline{X_3})-d_{G_2}(u,X_2)\ge \gamma^{1/8} n,$$
	which contradicts $K_3\not\subseteq G_1$, proving \ref{A4}.

	Fix an arbitrary $w\in X_1$, suppose to the contrary that $d(w,X_1)\ge \gamma^{1/18}n>d_{G_1}(w,X_1)+|J\cup J_1|$, due to~\eqref{eq-all2-X1} and \eqref{eq: J size} and \eqref{eq: Ji size}. Fix a vertex $u\in N_{G_2}(w,X_1)\setminus (J\cup J_1)$. By \ref{A4}, we may assume that $d_{G_1}(w,X_3)<\gamma^{1/9}n$. Then by \ref{X2} and the fact that $u\notin J\cup J_1$, we have
	$$|N_{G_2}(\{w,u\}, X_3)|\ge \mincr(G[X_1,X_2,X_3])-d_{G_1}(w,X_3)-d_{\overline{G}}(u,\overline{X_1})-d_{G_1}(u,X_3)> \alpha(G_2[X_3]),$$
	contradicting $K_4\not\subseteq G_2$. Thus, for each $i\in [3]$, we have $\Delta(G[X_i])\le \gamma^{1/18}n$, proving \ref{A3}. Consequently, 
	$$\mincr(G[X_1,X_2,X_3])\ge \delta(G)-\max_{i\in[3]}(\Delta(G[X_i])+|X_i|)\ge n/3-\gamma^{1/19} n,$$
	proving \ref{A5}. Together with \ref{A1}, this implies that
    for all $i\in \{2,3\}$ and $v\in X_i$, we have 
    $$d_{G_2}(v,X_1)\ge \mincr(G[X_1,X_2,X_3])-d_{G_1}(v,X_1)\stackrel{(\ref{eq-all2-X1})}{\ge}|X_1|-\gamma^{1/20}n,$$
	and that
	$$d_{G_1}(v,X_{5-i})\ge \mincr(G[X_1,X_2,X_3])-d_{G_2}(v,X_{5-i})\stackrel{(\ref{eq-X3-vx-pattern})}{\ge}|X_{5-i}|-\gamma^{1/20}n,$$
	proving \ref{A6} and the second part of \ref{P1} as desired.
\end{proof}



\subsection{Proof of Lemma~\ref{thm-K3K4-upperbound}} \label{sec-K3K4-upperbound}

Suppose that $e(G)> (\frac{1}{3}+ \frac{\delta}{2}+ \frac{3\delta^2}{2})n^2$.
By applying Lemma~\ref{lem: min degree} with $G,2/3,\delta+ 3\delta^2$ playing the roles of $G,d,\epsilon$, respectively, to obtain an $n'$-vertex graph $G'$ with $n'\geq \delta^{1/2} n/2$. Then $\delta(G')\geq 2n'/3$. Let $\delta':= \delta n/n' \in [\delta, \delta^{1/3}]$. Since $1\le \alpha(G')\leq \alpha(G) = \delta n = \delta' n'$,  we have
\begin{equation}\label{eq-Gprime-K3K4}
	e(G') \geq \frac{n'^2}{3} + \left(\frac{\delta}{2}+ \frac{3\delta^2}{2}\right)n^2-\frac{n-n'}{3}
	\geq \left(\frac{1}{3} + \frac{\delta'}{2} + \frac{3\delta'^2}{2}\right)n'^2.
\end{equation}

Note that $\phi$ still induces an edge-colouring of $G'$ which is $(K_3,K_4)$-free. As $1/n \ll \delta \ll \gamma$ and $n' \geq \delta^{1/2} n/2$ and $\delta' \in [\delta,\delta^{1/3}]$, we can apply Lemma~\ref{thm-K3K4-stab} with $G', \delta',\gamma$ playing the roles of $G,\delta,\gamma$ to obtain a partition $X_1\cup X_2\cup X_3$ of $V(G')$ satisfying \ref{A1}--\ref{A6}.
We assume that \ref{P1} occurs.\footnote{The~\ref{P2} case is only easier, we include its proof in the online arXiv version. In fact, graphs satisfying \ref{P2} case can only have at most $n'^2/3+\delta n'^2/2$ edges, a contradiction}. Define 
	$$A:=\{v\in X_1: d_{G'_1}(v,X_2)\ge n'/5 \}\quad \mbox{and} \quad B:=\{v\in X_1: d_{G'_1}(v,X_3)\ge n'/5 \}.$$
	Note that \ref{A4} implies that $A\cap B=\emptyset$. Bounding $e(G)$ amounts to show the following claim.
	\begin{claim}\label{cl-K3K4-triangle-free}
		The following hold:
		\begin{enumerate}[label={\rm (G$'$\arabic*)}]
		\item \label{G'1}For each $i\in\{2,3\}$, the graph $G'[X_i]$ is $K_3$-free.
		
		\item \label{G'2} Both $A$ and $B$ are independent sets and so $|A|,|B|\le \alpha(G')\leq \delta' n'$.
		
	\item \label{G'3} Both $G'[X_1\setminus A]$ and $G'[X_1\setminus B]$ are $K_3$-free.
		\end{enumerate}
	\end{claim}
	
First, we show how Claim~\ref{cl-K3K4-triangle-free} implies  Lemma~\ref{thm-K3K4-upperbound}. For each $i\in\{2,3\}$, \ref{G'1} implies that $\Delta(G'[X_i])\le\alpha(G')\le \delta' n'$, and so
	$e(G'[X_i])\le \delta' n'|X_i|/2$. On the other hand, 
\ref{G'2} and \ref{G'3} imply that $\Delta(G'[X_1\setminus A]),\Delta(G'[X_1\setminus B])\le \alpha(G')\le \delta' n'$, and so $e(G'[A,X_1\setminus (A\cup B)])\le \delta' n'|A|$. Therefore,
	\begin{eqnarray*}
		e(G'[X_1])&=& e(G'[X_1\setminus A])+e(G'[A,B])+e(G'[A,X_1\setminus(A\cup B)])\\
		&\le& (|X_1|-|A|)\delta' n'/2+|A|\delta' n'+\delta' n'|A|\le \delta' n'|X_1|/2+3\delta'^2 n'^2/2.
	\end{eqnarray*}
	Thus, we have
	$$e(G')\le e(G'[X_1,X_2,X_3])+\sum_{i\in[3]}e(G'[X_i])\le n'^2/3+\delta' n'^2/2+3\delta'^2n'^2/2,$$
	contradicting~\eqref{eq-Gprime-K3K4}. Thus we conclude that $e(G)\leq (\frac{1}{3}+ \frac{\delta}{2}+ \frac{3\delta^2}{2})n^2$.
	
	\begin{proof}[Proof of Claim~\ref{cl-K3K4-triangle-free}]
		Fix arbitrary $i\in \{2,3\}$. Suppose that $T=\{u,v,w\}$ induces a triangle in $G'[X_i]$. By \ref{A1} and \ref{P1}, we see that $|N_{G'_1}(T,X_{5-i})|\ge n/4$. As $G'_1$ is $K_3$-free, this implies that $T$ is monochromatic in colour $2$. But then, by \ref{A6}, we have $|N_{G'_2}(T,X_1)|\ge n/4$. This contradicts $K_4\not\subseteq G'_2$, proving \ref{G'1}.
		
		Suppose that $uv$ is an edge in $G'[A]$ (the proof for $B$ is similar). By the definition of $A$ and \ref{A1}, we have $|N_{G'_1}(uv, X_2)|\ge 2n/5-|X_2|\ge n/20$, implying that $\phi(uv)=2$ as $G'_1$ is $K_3$-free. Also by \ref{A4} and the fact that $u,v\in A$, we see that both $d_{G'_1}(u,X_3)$ and $d_{G'_1}(v,X_3)$ are less than $\gamma^{1/9}n$. Thus, by \ref{A5}, we have that
		$$|N_{G'_2}(uv, X_3)|\ge 2(\mincr(G'[X_1,X_2,X_3])-\gamma^{1/9}n)-|X_3|\ge n/4>\alpha(G'_2[X_3]).$$
		Hence, there exists an edge of colour $2$ in $N_{G'_2}(uv,X_3)$, contradicting $K_4\not\subseteq G'_2$, proving \ref{G'2}.
		
		Suppose $T=\{u,v,w\}$ induces a triangle in $X_1\setminus B$ (the proof for $X_1\setminus A$ is similar). Since $G'_1$ is $K_3$-free, we may assume that $\phi(uw)=2$. To prove \ref{G'3}, it suffices to show that $|N_{G'_2}(uw, X_i)|\ge n/30>\alpha(G'_2[X_i])$ for some $i\in\{2,3\}$, since then $K_4\subseteq G'_2$, a contradiction. As $A$ is an independent set due to \ref{G'2}, we may further assume that $w\notin A\cup B$. By the definition of $A$ and $B$, we have 
		\begin{equation}\label{eq-wXi-2deg}
		d_{G'_2}(w, X_i)\ge \mincr(G'[X_1,X_2,X_3])-n/5\ge n/10, \quad \forall~ i\in\{2,3\}.
		\end{equation}
		For each $i\in \{2,3\}$, let $W_i:=N_{G'_2}(w,X_i)$. Note that $d_{G'_1}(u,X_2)\ge d_{G'_1}(u,W_2)\ge n/20$, since otherwise $|N_{G'_2}(uw,X_2)|\ge n/30$ as desired. Then \ref{A4} implies that $d_{G'_1}(u,X_3)<\gamma^{1/9}n$. Together with \ref{A1},~\ref{A5}  and~\eqref{eq-wXi-2deg}, we have
		\begin{equation*}
		|N_{G'_2}(uw,X_3)|\ge \mincr(G'[X_1,X_2,X_3])-d_{G'_1}(u,X_3)+d_{G'_2}(w,X_3)-|X_3|\ge n/30,
		\end{equation*}
		as desired.
	\end{proof}
	This completes the proof of Lemma~\ref{thm-K3K4-upperbound}, providing the second equality of Theorem~\ref{thm-K3Ks}.


\section{Stability for $\varrho(K_3,K_3,\delta)$ without regularity}\label{sec-K3K3}
In this section, we present the upper bound on $\varrho(K_3,K_3,\delta)$.\footnote{The upper bound on $\varrho(K_3,K_5,\delta)$ can be proved by combining ideas in the proofs of the upper bounds on $\varrho(K_3,K_3,\delta)$ and $\varrho(K_3,K_4,\delta)$, we include its proof in the online arXiv version.} For convenience, we rephrase the upper bound as follows.

\begin{lemma}\label{thm-K3K3-weak}
Suppose $0<1/n \ll \delta < 10^{-13}$.
Let $G$ be an $n$-vertex $(K_3,K_3)$-free graph with $\alpha(G)\le \delta n$. Then $$e(G)\le \frac{n^2}{4}+\frac{\delta n^2}{2}.$$ 
\end{lemma}

We will prove Lemma~\ref{thm-K3K3-weak} using the following coloured stability.
\begin{lemma}\label{thm-2col-tri-stab}
Suppose $0<1/n \ll \delta < 10^{-6}$.
	Let $G$ be an $n$-vertex $(K_3,K_3)$-free graph with $\alpha(G)\le \delta n$ and $\delta(G)\ge n/2$. Then for any $(K_3,K_3)$-free $\phi: E(G)\rightarrow [2]$, there exists a partition $V(G)=A\cup B$ with $|A|,|B|=n/2\pm \delta^{1/3}n$ and $\delta(G_1[A,B])\ge 2n/5$.
\end{lemma}

We will present a proof of Lemma~\ref{thm-2col-tri-stab} without regularity lemma. First, we show how it implies Lemma~\ref{thm-K3K3-weak}.
\begin{proof}[Proof of Lemma~\ref{thm-K3K3-weak}]
Suppose that $e(G)> (\frac{1}{4}+ \frac{\delta}{2})n^2$.
By applying Lemma~\ref{lem: min degree} with $G,1/2,\delta$ playing the roles of $G,d,\epsilon$, respectively, to obtain an $n'$-vertex graph $G'$ with $n'\geq \delta^{1/2} n/2$ satisfying the following, where $\delta':= \delta n/n' \in [\delta, 10^{6})$:
$$\delta(G')\geq n'/2 \enspace \text{and} \enspace e(G') \geq \frac{n'^2}{4} + \frac{\delta n^2}{2}-\frac{n-n'}{4}
\geq \left(\frac{1}{4} + \frac{\delta'}{2}\right)n'^2.$$
Moreover, $\alpha(G')\leq \alpha(G) = \delta n = \delta' n'$.

As $G$ is $(K_3,K_3)$-free, $G'\subseteq G$ is also $(K_3,K_3)$-free and there exists an $(K_3,K_3)$-free $2$ edge-colouring $\phi$ of $G'$.
As $1/n \ll \delta$ and $n'\geq \delta^{1/2} n/2$ and $\delta' < 10^{-6}$, we apply Lemma~\ref{thm-2col-tri-stab} to obtain a partition $A\cup B$ of $V(G')$ with $|A|,|B|=n'/2\pm \delta'^{1/3}n$ and $\delta(G'_1[A,B])\ge 2n'/5$.

We claim that both $G'[A]$ and $G'[B]$ are triangle-free. Suppose that $T\in {A\choose 3}$ induces a triangle in $G'[A]$. Since $G'_1$ is $K_3$-free and $|N_{G'_1}(T,B)|\ge 3\delta(G'_1[A,B])-2|B|>n/6$, we see that no edge in $T$ can be of colour $1$. But then $T$ is monochromatic in colour $2$, contradicting $K_3\not\subseteq G'_2$. Thus, $e(G'[A])\le \Delta(G'[A])|A|/2\le \alpha(G')|A|/2\le \delta' n'|A|/2$. Similarly, $e(G'[B])\le \delta' n'|B|/2$. Hence,
    $$e(G')=e(G'[A,B])+e(G'[A])+e(G'[B])\le n'^2/4+\delta' n'^2/2,$$
 a contradiction. Thus we conclude $e(G)\leq (\frac{1}{4}+ \frac{\delta}{2})n^2$.
\end{proof}

\begin{proof}[Proof of Lemma~\ref{thm-2col-tri-stab}]
	Assume without loss of generality that $e(G_1)\ge e(G_2)$, so $e(G_1)\ge n^2/8$. It is easy to see that the following fact follows from $G$ being $(K_3,K_3)$-free and that $\alpha(G)\le \delta n$.
	\begin{fact}\label{fact-tri}
		For all $x,y\in V(G)$, we have $|N_{G_1}(x)\cap N_{G_2}(y)|\le \alpha(G)\le \delta n$.
	\end{fact}
	We will sequentially choose four vertices as follows. 
	\begin{itemize}
		\item Take a vertex $x$ with maximum $G_1$-degree, i.e.~$d_{G_1}(x)=\Delta(G_1)$, and set $X:=N_{G_1}(x)$. Then $|X|\ge \frac{2e(G_1)}{n}\ge n/4$. 
		\item Choose a vertex $y\in X$ with maximum $G_1$-degree and set $Y:=N_{G_1}(y)$. Note that as $G_1$ is $K_3$-free, $X\cap Y=\emptyset$. Denote $Z:=V(G)\setminus (X\cup Y)$ and $\alpha:= |Z|/n$. 
		\item Pick now $x'\in Z$ with maximum $G_2$-degree in $Z$. Let $X':=N_{G_2}(x',Z)$ and $\beta := |X'|/n$. By definition, we have $0\le \beta\le \alpha$.
		\item Finally, take $y'\in X'$ with maximum $G_2$-degree in $Z$ and set $Y':=N_{G_2}(y',Z)$. Similarly, as $G_2$ is $K_3$-free, $X'\cap Y'=\emptyset$. So $|Y'|\le |Z\setminus X'|=(\alpha-\beta)n$.
	\end{itemize}
	
	\begin{claim}\label{cl-al-23}
		We have $|X|+|Y|\ge n/3$, consequently, $\alpha\le 2/3$.
	\end{claim}
	\begin{proof}
		We may assume that $|X|\le n/3$ otherwise we are done. By the definition of $x$ and $X$, every vertex in $X$ (resp. not in $X$) has $G_1$-degree at most $|Y|$ (resp. $|X|$). Thus,
		\begin{equation}\label{eq-eG1-2}
			\frac{n^2}{4}\le 2e(G_1)\le \sum_{u\notin X}d_{G_1}(u)+\sum_{v\in X}d_{G_1}(v)\le (n-|X|)|X|+|X||Y|,
		\end{equation}
		implying that $f(a):=(1-a)a+ab-\frac{1}{4}\ge 0$, where $1/4\le a:=|X|/n\le 1/3$ and $b:=|Y|/n$. As $f'(a)=1+b-2a>0$, we have $0\le f(a)\le f(1/3)$, implying that $b\ge 1/12$ and then $a+b\ge 1/3$ as desired.
	\end{proof}
	
	Let us show the following bound on the size of $G_1$:
	\begin{equation}\label{eq-eG1}
	e(G_1)\le \frac{(1-\beta)^2n^2}{4}+\delta n^2.
	\end{equation}
	Indeed, the first term above bounds $e(G_1[\overline{X'}])$ as $G_1$ is $K_3$-free; while the second term bounds all $G_1$-edges with at least one endpoints in $X'$. To see this, for each vertex $v\in V(G)$, we have $d_{G_2}(v,X')\subseteq  N_{G_1}(v)\cap N_{G_2}(x')$, thus the desired bound follows from Fact~\ref{fact-tri}. Similarly, we can bound all $G_2$-edges with at least one endpoints in $X\cup Y$ by $2\delta n^2$. Thus, we have that $e(G_2)\le e(G_2[Z])+2\delta n^2$.
		\begin{claim}\label{cl-G2edge}
		We have that $e(G_2)\le 22\delta n^2$.
	\end{claim}
	\begin{proof}
		As we have $e(G_2)\le e(G_2[Z])+2\delta n^2$, it suffices to prove $e(G_2[Z])\le 20\delta n^2$. Note first that, by the definition of $x'$, we have that
		\begin{equation}\label{eq-eG2-ab-over2}
			e(G_2[Z])\le \frac{|X'|\cdot |Z|}{2}=\frac{\alpha \beta n^2}{2}.
		\end{equation}
		On the other hand, analogous to~\eqref{eq-eG1-2}, by the definition of $y'$, we have
		\begin{eqnarray}\label{eq-eG2-a-b}
		e(G_2[Z])&\le& \frac{1}{2}\left(\sum_{u\in Z\setminus X'}d_{G_2}(u,Z)+\sum_{v\in X'}d_{G_2}(v,Z)\right)\nonumber\\
		&\le& \frac{1}{2}\left(|Z\setminus X'|\cdot\beta n+|X'|\cdot (\alpha-\beta)n\right)=(\alpha-\beta)\beta n.
		\end{eqnarray}

    We now distinguish two cases.
    
    \case{1}: $\beta\le \alpha/2$. By~\eqref{eq-eG1} and~\eqref{eq-eG2-ab-over2}, we have that
    \begin{equation}\label{eq-beta-small}
    	\frac{n^2}{4}\le e(G)\le \frac{n^2}{4}((1-\beta)^2+2\alpha\beta)+3\delta n^2 \quad \Leftrightarrow \quad (\beta^2-2\beta+2\alpha\beta)+3\delta \ge 0.
    \end{equation}
    Let $f(\beta):=\beta^2-2\beta+2\alpha\beta+3\delta \ge 0$, then $f'(\beta)=2\beta-2+2\alpha\le 3\alpha-2\le 0$ as $\alpha \le 2/3$ due to Claim~\ref{cl-al-23}. So we have $\beta\le 40\delta $, as otherwise
    $$f(\beta)\le f(40\delta )=40\delta (40\delta -2+2\alpha)+3\delta<0,$$
    contradicting~\eqref{eq-beta-small}. Then by~\eqref{eq-eG2-ab-over2}, we have $e(G_2[Z])\le \beta n^2/2\le 20\delta n^2$ as desired.
    
    \case{2}: $\beta\ge \alpha/2$. By~\eqref{eq-eG1} and~\eqref{eq-eG2-a-b}, we have that
    \begin{equation}\label{eq-beta-big}
    \frac{n^2}{4}\le e(G)\le \frac{n^2}{4}((1-\beta)^2+4(\alpha-\beta)\beta)+3\delta n^2 \quad \Leftrightarrow \quad (4\alpha\beta-2\beta-3\beta^2)+3\delta \ge 0.
    \end{equation}
    Let $g(\beta):=4\alpha\beta-2\beta-3\beta^2+3\delta \ge 0$, then $g'(\beta)=4\alpha-2-6\beta\le 4\alpha-2-3\alpha< 0$. So we have $\alpha\le 20\delta $, as otherwise, using that $\alpha\le 2/3$,
    $$g(\beta)\le g(\alpha/2)=\frac{\alpha}{2}\left(4\alpha-2-3\cdot\frac{\alpha}{2}\right)+3\delta \le -\frac{\alpha}{6}+3\delta <0,$$
    contradicting~\eqref{eq-beta-big}. Then by~\eqref{eq-eG2-ab-over2}, we have $e(G_2[Z])\le \alpha n^2/2\le 10\delta n^2$ as desired.
    \end{proof}
  By Claim~\ref{cl-G2edge}, we have $e(G_1)\ge n^2/4-22\delta n^2$. Apply Theorem~\ref{thm-furedi-stab} to $G_1$ with $t=22\delta n^2$ and let $V(G)=V(G_1)=A\cup B$ be an arbitrary max-cut partition of $G_1$. Then we have
  \begin{equation}\label{eq-G1-AB-dense}
  e(G[A])+e(G[B])\le 3t=66\delta n^2 ~ \Rightarrow ~ e(G_1[A,B])\ge e(G_1)-66\delta n^2\ge \frac{n^2}{4}-88\delta n^2,
  \end{equation}
  and $|A|,|B|=n/2\pm 2\sqrt{t}=n/2\pm 10\sqrt{\delta }n$. Note that there exists a vertex $v\in B$ with
  $$d_{G_1}(v,A)\ge \frac{e(G_1[A,B])}{|B|}\ge\frac{n^2/4-88\delta n^2}{n/2+10\sqrt{\delta }n}\ge \frac{n}{2}- 10\sqrt{\delta }n\ge |A|-20\sqrt{\delta }n.$$
  Consequently, for any $u\in V(G)$, we have $d_{G_2}(u,A)\le 21\sqrt{\delta }n$, as otherwise $|N_{G_2}(u)\cap N_{G_1}(v)|\ge\sqrt{\delta }n$, contradicting Fact~\ref{fact-tri}. Similarly, $d_{G_2}(u,B)\leq 21\sqrt{\delta}n$. Thus we have
  $\Delta(G_2)\le 42\sqrt{\delta }n$.
  
  We claim that $A\cup B$ is the desired partition. Suppose $d_{G_1}(w,B)<2n/5$ for some $w\in A$. Then $d_{G_1}(w,A)\ge \delta(G)-\Delta(G_2)-d_{G_1}(w,B)\ge n/20$.
  As $A\cup B$ is a max-cut, we see that
  $d_{G_1}(w,B)\ge d_{G_1}(w,A)-\Delta(G_2)\ge n/30$.
  Since $G_1$ is $K_3$-free, there is no edge of $G_1$ in $[N_{G_1}(w,A),N_{G_1}(w,B)]$, implying that $e(G_1[A,B])\le n^2/4-(n/20)\cdot (n/30)$, contradicting~\eqref{eq-G1-AB-dense} and that $\delta < 10^{-6}$.
\end{proof}

\section{Proof of Theorem~\ref{thm-3q-sm-ind}}\label{sec-3q}
\subsection{Upper bound}
Let $s\ge 2$ and fix a function $g_s(n)$ satisfying \eqref{eq-defn-sublinear}.
Note that a function $g'_s(n)$ satisfying $ g_s(n) = (g'_s(n)/n)^2 n$ is also a function satisfying \eqref{eq-defn-sublinear}.
We choose constants such that 
$0<1/n \ll 1/M' \ll \epsilon \ll \gamma \ll 1$. In particular, $1/g'_s(n)\ll 1/M'$.

Let $G$ be an $n$-vertex graph with $\alpha(G)\le g_s(n)$ with a $2$-edge-colouring $\phi$. 

We apply Lemma~\ref{lem-split-2-col} with $c=(g'_s(n)/n)^2$, to obtain a partition $V_1^*\cup V_2^*$ such that $\alpha(G_i[V_i^*])\leq g'_s(n)$.
 Apply Theorem~\ref{lm:RL} with $G, V^*_1, V^*_2,\phi, \eps, \epsilon^{-1}$ and $M'$ playing the roles of $G,U_1,U_2,\phi, \epsilon,M$ and $M'$ to obtain an $\eps $-regular partition $V_1\cup \dots\cup V_m$ with $\epsilon^{-1}\leq m\leq M'$ which refines the partition $V_1^*\cup V_2^*$. Let $R:=R(\eps ,\gamma,\phi,(V_i)_{i\in[m]})$ be its reduced graph. Let the colouring $\phiind$ be as defined in the proof of Lemma~\ref{thm-K3K4-stab}. 
So if $\phiind(i)=j$ for some $i\in V(R)$ and $j\in[2]$, it means the corresponding cluster $V_i$ in $G$ satisfies $\alpha(G_j[V_i])\le g'_s(n)$. 

By Tur\'an's Theorem, it suffices to show the following.
\begin{enumerate}[label={\rm (R\arabic*)}]
\item\label{R1} $R$ is $K_{R(3,s)}$-free if $\phi$ is $(K_3,K_{2s-1})$-free; 
\item\label{R2} $R$ is $K_{R(3,s)+1}$-free if $\phi$ is $(K_3,K_{2s})$-free.
\end{enumerate}
Indeed, it is easy to see that \ref{R1} implies
$e(G)\leq \frac{1}{2} \left(1-\frac{1}{R(3,s)-1}\right)n^2+3\gamma n^2$ and
\ref{R2} implies $e(G)\leq \frac{1}{2} \left(1-\frac{1}{R(3,s)}\right)n^2+3\gamma n^2$ .

To show \ref{R1} and \ref{R2}, without loss of generality, assume that $[t]\subseteq V(R)$ induces a maximum size clique in $R$. As the case $s=2$ is covered in Theorem~\ref{thm-K3Ks}, we assume that $s\ge 3$.

Suppose that $G$ is $(K_3,K_{2s-1})$-free (resp. $(K_3,K_{2s})$-free). Suppose that $\phiind(i)=2$ for all $i\in[t]$, then by Lemma~\ref{lem-drc-j}, $\phiind|_{[t]}$ is $(K_3,K_s)$-free, and so $t\le R(3,s)-1$ as desired. We may then assume that $\phiind(t)=1$. Then $\phiind(it)=2$ for all $i\in[t-1]$, as otherwise it is easy to see that one can embed $K_3$ in $G_1[V_i\cup V_t]$. Consequently, by Lemma~\ref{lem-drc-j}, we have that $\phiind|_{[t-1]}$ is $(K_3,K_{s-1})$-free (resp. $(K_3,K_{s})$-free). Hence, $t-1\le R(3,s-1)-1\le R(3,s)-2$ (resp. $t-1\le R(3,s)-1$) as desired.

\subsection{Lower bound}
Let $n$ be a sufficiently large number, and
let $H(n)$ be an $n$-vertex $K_3$-free graph with independence number $O(\sqrt{n\log n})$. The celebrated result of Kim \cite{KimRamsey} shows the existence of such graphs.

\subsubsection{Lower bound for $\rt(n,K_3,K_{2s-1},g_s(n))$} Let $t=R(3,s)-1$ and $\phi: {[t]\choose 2}\rightarrow [2]$ be a $(K_3,K_s)$-free colouring. Let $G$ be obtained from adding a copy of $H(n/t)$ to each partite set of $T_t(n)$. The following colouring witnesses $G$ being $(K_3,K_{2s-1})$-free: colour all edges inside each partite set colour $2$ and colour all crossing edges according to $\phi$, i.e.~for any $ij\in{[t]\choose 2}$ and $h\in[2]$, all edges in $[X_i,X_j]$ are of colour $h$ if $\phi(ij)=h$.
\subsubsection{Lower bound for $\rt(n,K_3,K_{2s},g_s(n))$} Let $t=R(3,s)$ and $\phi: {[t-1]\choose 2}\rightarrow [2]$ be a $(K_3,K_s)$-free colouring. Let $G$ be obtained from adding a copy of $H(n/t)$ to each partite set of $T_t(n)$. The following colouring witnesses $G$ being $(K_3,K_{2s})$-free: colour all edges inside $X_t$ colour $1$, and edges inside $X_i$ colour $2$ for all $i\in[t-1]$; colour all crossing edges in $[X_1,\ldots,X_{t-1}]$ according to $\phi$ and colour all $[X_i,X_t]$-edges colour $2$ for all $i\in [t-1]$.  
\section{Concluding remarks}\label{sec-remarks}

\subsection{The value of $\varrho(K_3,K_6,\delta)$}
We conjecture that the following equality holds.
$$\varrho(K_3,K_6,\delta )=\frac{5}{12}+\frac{\delta}{2}+2\delta^2.$$
The lower bound is given by the construction below, see Figure~\ref{fig:estimable}.
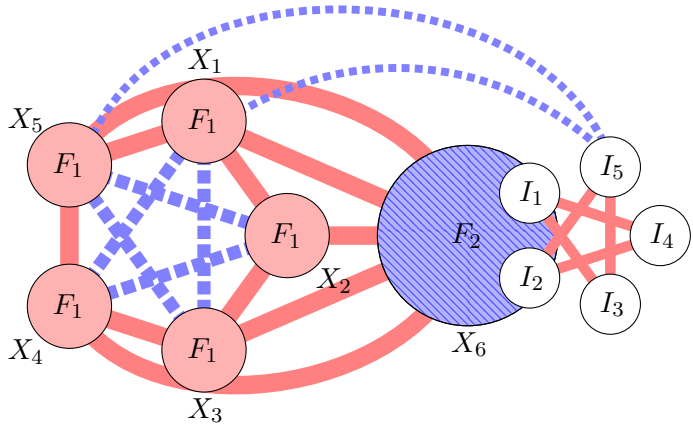
\begin{figure}
	\centering
	\begin{tikzpicture}[scale = 0.8]
	\tikzset{be/.style = {line width=5pt, color=blue!50!white, dotted}}
	\tikzset{re/.style = {line width=7pt, color=red!50!white}}
	\tikzset{b2e/.style = {line width=3pt, color=blue!50!white, dotted}}
	\tikzset{r2e/.style = {line width=4pt, color=red!50!white}}

	\draw[re] (1*2-2,0) to (0.309*2-2,0.951*2);
	\draw[re] (0.309*2-2,0.951*2) to (-0.8090*2-2,0.5878*2);
	\draw[re] (-0.8090*2-2,0.5878*2) to (-0.8090*2-2,-0.5878*2);
	\draw[re] (-0.8090*2-2,-0.5878*2) to (0.309*2-2,-0.951*2);
	\draw[re] (0.309*2-2,-0.951*2) to (1*2-2,0);
	
	\draw[be] (1*2-2,0) to  (-0.8090*2-2,0.5878*2);
	\draw[be] (0.309*2-2,0.951*2)  to (-0.8090*2-2,-0.5878*2);
	\draw[be] (-0.8090*2-2,0.5878*2) to (0.309*2-2,-0.951*2);
	\draw[be] (-0.8090*2-2,-0.5878*2) to (1*2-2,0);
	\draw[be] (0.309*2-2,-0.951*2) to (0.309*2-2,0.951*2);
	
	\draw[re] (3,0) to (0.309*2-2,0.951*2);
	\draw[re] (3,0) to[bend right=70] (-0.8090*2-2,0.5878*2);
	\draw[re] (3,0) to[bend left=70] (-0.8090*2-2,-0.5878*2);
	\draw[re] (3,0) to (0.309*2-2,-0.951*2);
	\draw[re] (3,0) to (1*2-2,0);

	\draw[b2e] (0.309*1.2+0.15+5,0.951*1.2) to[bend right=40] (0.309*2+0.25-2+0.1,0.951*2);
	\draw[b2e] (0.309*1.2+0.05+5,0.951*1.2) to[bend right=70] (-0.8090*2+0.25-2+0.1,0.5878*2+0.4);
			
	\draw[fill=red!30] (1*2-2,0) ellipse (0.7cm and 0.7cm);
	\node [rectangle, fill=none, draw=none, text width = 1cm] at (1*2+0.25-2+0.1,0) {$F_1$};
	\node [rectangle, fill=none, draw=none, text width = 1cm] at (1*2+0.25-2+0.1+0.75,0-0.75) {$X_2$};
	
	\draw[fill=red!30] (0.309*2-2,0.951*2) ellipse (0.7cm and 0.7cm);
	\node [rectangle, fill=none, draw=none, text width = 1cm] at (0.309*2+0.25-2+0.1,0.951*2) {$F_1$};
	\node [rectangle, fill=none, draw=none, text width = 1cm] at (0.309*2+0.25-2+0.1,0.951*2+1) {$X_1$};
	
	\draw[fill=red!30] (-0.8090*2-2,0.5878*2) ellipse (0.7cm and 0.7cm);
	\node [rectangle, fill=none, draw=none, text width = 1cm] at (-0.8090*2+0.25-2+0.1,0.5878*2) {$F_1$};
	\node [rectangle, fill=none, draw=none, text width = 1cm] at (-0.8090*2+0.25-2+0.1-0.75,0.5878*2+0.75) {$X_5$};
	
	\draw[fill=red!30] (-0.8090*2-2,-0.5878*2) ellipse (0.7cm and 0.7cm);
	\node [rectangle, fill=none, draw=none, text width = 1cm] at ((-0.8090*2+0.25-2+0.1,-0.5878*2) {$F_1$};
	\node [rectangle, fill=none, draw=none, text width = 1cm] at ((-0.8090*2+0.25-2+0.1-0.75,-0.5878*2-0.75) {$X_4$};
	
	\draw[fill=red!30] (0.309*2-2,-0.951*2) ellipse (0.7cm and 0.7cm);
	\node [rectangle, fill=none, draw=none, text width = 1cm] at (0.309*2+0.25-2+0.1,-0.951*2) {$F_1$};
	\node [rectangle, fill=none, draw=none, text width = 1cm] at (0.309*2+0.25-2+0.1,-0.951*2-1) {$X_3$};
	
	\draw[fill=blue!30] (3,0) ellipse (1.5cm and 1.5cm);
	\draw[pattern=north west lines, pattern color=blue!70] (3,0) ellipse (1.5cm and 1.5cm);
	\draw[r2e] (1*1.2+5,0) to  (-0.8090*1.2+5,0.5878*1.2);
	\draw[r2e] (0.309*1.2+5,0.951*1.2) to  (-0.8090*1.2+5,-0.5878*1.2);
	\draw[r2e]  (-0.8090*1.2+5,0.5878*1.2) to  (0.309*1.2+5,-0.951*1.2);
	\draw[r2e]  (-0.8090*1.2+5,-0.5878*1.2) to  (1*1.2+5,0);
	\draw[r2e] (0.309*1.2+5,-0.951*1.2) to (0.309*1.2+5,0.951*1.2);

	\node [rectangle, fill=none, draw=none, text width = 1cm] at (3+0.25+0.1,0) {$F_2$};
		\node [rectangle, fill=none, draw=none, text width = 1cm] at (3+0.25+0.1,0-1.8) {$X_6$};
	
	\draw[fill=white] (1*1.2+5,0) ellipse (0.5cm and 0.5cm);
	\node [rectangle, fill=none, draw=none, text width = 1cm] at (1*1.2+0.35+5+0.1,0) {$I_4$};
	
	\draw[fill=white] (0.309*1.2+5,0.951*1.2) ellipse (0.5cm and 0.5cm);
	\node [rectangle, fill=none, draw=none, text width = 1cm] at (0.309*1.2+0.35+5+0.1,0.951*1.2) {$I_5$};
	
	\draw[fill=white] (-0.8090*1.2+5,0.5878*1.2) ellipse (0.5cm and 0.5cm);
	\node [rectangle, fill=none, draw=none, text width = 1cm] at (-0.8090*1.2+0.35+5+0.1,0.5878*1.2) {$I_1$};
	
	\draw[fill=white] (-0.8090*1.2+5,-0.5878*1.2) ellipse (0.5cm and 0.5cm);
	\node [rectangle, fill=none, draw=none, text width = 1cm] at ((-0.8090*1.2+0.35+5+0.1,-0.5878*1.2) {$I_2$};

	\draw[fill=white] (0.309*1.2+5,-0.951*1.2) ellipse (0.5cm and 0.5cm);
	\node [rectangle, fill=none, draw=none, text width = 1cm] at (0.309*1.2+0.35+5+0.1,-0.951*1.2) {$I_3$};

	\end{tikzpicture}
	\caption{A graph with no blue (dotted) $K_3$ and no red $K_6$. 
	All edges incident to $\bigcup_{i\in [5]}I_i$ and $\bigcup_{i\in [5]}X_i$ are omitted in the picture except blue edges between $I_5$ and $X_5\cup X_1$.
	}\label{fig:estimable}
\end{figure}

\begin{itemize}
	\item Let $F_1:=F(\frac{n}{6}, d_1)$ and $F_2:=F(\frac{n}{6}-\frac{3\delta n}{2}, d_2)$ where $d_i\in [\delta n- o(n),\delta n]$. So $e(F_1)=\frac{\delta n^2}{12}\pm o(n^2)$ and $e(F_2)=\frac{\delta n^2}{12}-\frac{3\delta^2n^2}{4}\pm o(n^2)$.
	
	\item Let $I=\{v_1,v_2,\ldots, v_{d_2}\}$ be an independent set of size $d_2$ in $F_2$. Let $I=I_1\cup I_2$ be an equipartition of $I$. Let $F$ be an $n/6$-vertex graph obtained from $F_2$ by 
	\begin{itemize}
		\item first adding 3 clone sets of $I_1$, say $I_i$ with $i\in\{3,4,5\}$; 
		\item adding all $[I_i,I_{i+2}]$-edges for each $i\in[5]$ (addition modulo $5$); and 
		\item adding an additional set of $\frac{3}{2}(\delta n-d_2)$ isolated vertices.
	\end{itemize}
	Note that $F$ is \emph{not} triangle-free, and
	$$e(F)=e(F_2)+\frac{3d_2}{2}\cdot d_2+5\left(\frac{d_2}{2}\right)^2=\frac{\delta n^2}{12}+2\delta^2n^2\pm o(n^2).$$
	\item Finally, let $G$ be the graph obtained from $T_6(n)$, by putting a copy of $F$ in $X_6$ and a copy of $F_1$ in $X_i$ for each $i\in[5]$. 
\end{itemize}
It is clear that $G$ has the desired size and easy to check that the following 2-edge-colouring $\phi$ of $G$ is $(K_3,K_6)$-free: 
\begin{itemize}
	\item let $\phi(X_i,X_{i+2})=1$ for each $i\in[5]$ (addition modulo $5$); 
	\item let $\phi(I_i,X_i\cup X_{i+1})=1$  for each $i\in[5]$ (addition modulo $5$); 
	\item let $\phi(e)=1$, for all $e\in E(G[X_6]\setminus G[\cup_{i\in[5]}I_i])$; 
	\item all other edges are of colour $2$. 
\end{itemize}

\subsection{The value of $\varrho(K_3,K_{2s})$}
	Recall that for triangle versus odd cliques, Erd\H os, Hajnal, Simonovits, S\'os and Szemer\'edi~\cite{EHSSSz} conjectured that $\varrho(K_3,K_{2s-1})$ is achieved by the $(R(3,s)-1)$-partite Tur\'an graph, i.e.~$\varrho(K_3,K_{2s-1})=\frac{1}{2}\left(1-\frac{1}{R(3,s)-1}\right)$. Base on Theorem~\ref{thm-3q-sm-ind}, we put forward the following conjecture for triangle versus even cliques.
	\begin{conjecture}
		For all $s\ge 2$, $\varrho(K_3,K_{2s})=\frac{1}{2}\left(1-\frac{1}{R(3,s)}\right)$.
	\end{conjecture}

\subsection{Ramsey-Tur\'an number with more than 2 colours}
We remark that the multicolour Ramsey-Tur\'an number for \emph{triangles} is related to a version of Ramsey number studied by Liu, Pikhurko and Sharifzadeh~\cite{LPSh}. They introduced $r^*(K_{a_1},\ldots, K_{a_k})$ as the largest integer $N$ such that there exists a colouring $\phi: {[N]\choose \le 2}\rightarrow [k]$ with the following property: 

\begin{itemize}
	\item[$(\ast)$] for each $i\in[k]$, there is no edge-monochromatic $K_{a_i}$ in colour $i$, and there is no edge incident to a vertex with the same colour, i.e.~$\phi(ij)\neq \phi(i)$ for any $j\neq i$.
\end{itemize}
Note that when an $n$-vertex graph $G$ is $(K_3,\ldots,K_3)$-free with $\alpha(G)=o(n)$, then the colouring $\phiind$ on its reduced graph $R$ satisfies~$(\ast)$, hence
$$\varrho(K_3,\ldots,K_3)=\frac{1}{2}\left(1-\frac{1}{r^*(K_3,\ldots,K_3)}\right).$$
In particular, Theorems~1.8 and~1.9 in~\cite{LPSh} imply $\varrho(K_3,K_3,K_3)=\frac{2}{5}$ and $\varrho(K_3,K_3,K_3,K_3)=\frac{15}{32}$.

In general Ramsey-Tur\'an numbers for larger cliques are \emph{not} determined by $r^*$, for example $\varrho(K_3,K_5)=\frac{2}{5}\neq \frac{1}{2}\left(1-\frac{1}{r^*(K_3,K_5)}\right)=\frac{3}{8}$ and $\varrho(K_4,K_4)=\frac{11}{28}\neq \frac{1}{2}\left(1-\frac{1}{r^*(K_4,K_4)}\right)=\frac{1}{3}$.

\medskip

{\footnotesize \obeylines \parindent=0pt
	
	\begin{tabular}{lllll}
		Jaehoon Kim	&\ & Younjin Kim	&\ &	Hong Liu        \\
		School of Mathematics &\ & Institute of Mathematical Sciences &\ & Mathematics Institute		  		 	 \\
		University of Birmingham &\ & Ewha Womans University &\ & University of Warwick 	  			 	 \\
		Birmingham &\ & Seoul &\ & Coventry                             			 \\
		UK &\ & South Korea &\ & UK						      \\
	\end{tabular}
}

\begin{flushleft}
	{\it{E-mail addresses}:
		\tt{j.kim.3@bham.ac.uk, younjinkim@ewha.ac.kr, h.liu.9@warwick.ac.uk.}}
\end{flushleft}

\appendix
\section{}

\subsection{When \ref{P2} occurs in $(K_3,K_4)$ stability}
We bounded $\varrho(K_3,K_4,\delta)$ from above in Section~\ref{sec-K3K4-upperbound} assuming that \ref{P1} occurs in the stability. Here we sketch a proof that if $G'$ satisfies \ref{P2}, then we have $e(G')\leq n'^2/3+\delta' n'^2/2$.
	
Recall that  we have a partition $X_1\cup X_2\cup X_3$ of $V(G')$ satisfying \ref{A1}--\ref{A6}. Assume that \ref{P2} occurs. It suffices to show that for each $i\in [3]$, the graph $G'[X_i]$ is triangle-free. Indeed, this would implies that for each $i\in[3]$, we have $\Delta(G'[X_i])\le \alpha(G')\le \delta' n'$, and then we have
\begin{equation}\label{eq-K3-free-cal}
	e(G')\le e(G'[X_1,X_2,X_3])+\sum_{i\in[3]}e(G'[X_i])\le \frac{n'^2}{3}+\frac{\delta' n'^2}{2},
\end{equation}	
as desired.

By symmetry, suppose to the contrary that $\{u,v,w\}\in {X_1\choose 3}$ induces a copy of $K_3$. As $G'_1$ is $K_3$-free, we may assume that $uv$ is of colour $2$. As $\alpha(G'_1[X_i])\le\gamma^{1/4}n'$ for each $i\in\{2,3\}$ and $G'_1$ is $K_3$-free, neither $u$ nor $v$ can have more than $\gamma^{1/5}n'$ neighbours in $X_2$ or $X_3$ in the graph $G'_1$. Then by \ref{A1}, \ref{A5} and \ref{P2}, for each $i\in \{2,3\}$, we can easily find, a vertex  $x_i\in N_{G'_2}(uv,X_i)$ such that $x_2x_3\in E(G'_2)$. Then $u,v,x_2,x_3$ induces a $K_4$ in $G'_2$, a contradiction.

\subsection{Upper bound for $\varrho(K_3,K_5,\delta)$}
Let $G$ be an $n$-vertex $(K_3,K_5)$-free graph with $\alpha(G)\le \delta n$. 
We first prove the following coloured stability using the approach for Lemma~\ref{thm-K3K4-stab}.
\begin{lemma}\label{thm-K3K5-stab}
Suppose $0< 1/n\ll \delta\ll\gamma\ll\eta\ll 1$. Let $G$ be an $n$-vertex $(K_3,K_5)$-free graph with $\alpha(G)\le \delta n$ and $\delta(G)\ge 4n/5$. Then for any $(K_3,K_5)$-free $\phi: E(G)\rightarrow [2]$, there exists a partition $V(G)=X_1\cup \ldots\cup X_5$ such that for each $i\in[5]$,
\begin{enumerate}[label={\rm (A$'$\arabic*)}]
	\item \label{A'1} $|X_i|=n/5\pm \eta n$;
	\item \label{A'2} $\alpha(G_2[X_i])\le \eta n$;
	\item \label{A'3} for every $v\in X_i$, we have  $d_{G_1}(v,X_{i+2})\ge |X_{i+2}|-\eta n$;
	\item \label{A'4} for every $v\in X_i$, we have$d_{G_2}(v,X_{i+1})\ge |X_{i+1}|-\eta n$.
\end{enumerate}
\end{lemma}
\begin{proof}[Sketch of proof]
	Let $G,\phi,R,R', \phiind$ be obtained analogous to that in the proof of Lemma~\ref{thm-K3K4-stab}. Then $R$ is close to $T_5(|R|)$ and $(Z_3,Z_5)$-free. Let BLUE and RED be colour $1$ and colour $2$ respectively. Let $U_1\cup \ldots\cup U_5$ be a max-cut 5-partition of $V(R)$. 
	As $R$ is close to $T_5(|R|)$, similar to~\eqref{eq-Ui-property} and~\eqref{eq-Ui-mincr}, we get that $\sum_{i\in[5]}e(R[U_i])=o(m^2)$, $|U_i|=m/5\pm o(m)$, and $\mincr(R[U_1,\ldots,U_5])\ge (\delta(R)-3\max|U_i|)/2\ge m/11$.

	We now show that each $U_i$ should be monochromatic in $\phiind$ with RED colour. Suppose this is not the case. As in Claim~\ref{cl-R-Vi-mono}, let $u\in X_1$ be a BLUE vertex and $x_2\in X_2$ be a typical vertex such that for each $i\in \{3,4,5\}$ we have $|N_{R}(u,X_i)\cap N_{R'}(x_2,X_i)|\geq m/20$.
As $R$ is close to $T_5(|R|)$, we can find three vertices $x_3\in X_3, x_4\in X_4, x_5\in X_5$ such that $\{u,x_2,x_3,x_4,x_5\}$ forms a copy of $K_5$ in $R$. Moreover, all edges incident to $x_2$ is also in $R'$.

If $w_i$ is BLUE for some $i\in [5]\setminus \{1\}$, then any edges in the $K_5$ incident to $u$ or $w_i$ are all RED, as otherwise we obtain a BLUE generalised clique $Z_3$. Let $\{j,k,\ell\}= [5]\setminus \{1,i\}$.
If all three remaining edges $x_{j}x_{k},x_{j}x_{\ell}, x_{k}x_{\ell}$ are RED, then we obtain a RED $K_5$ (which is also a generalised clique $Z_5$), a contradiction.
If all three edges $x_{j}x_{k},x_{j}x_{\ell}, x_{k}x_{\ell}$ are BLUE, then we obtain a BLUE $K_3$ (which is also a copy of $Z_3$), a contradiction.
Thus we may assume that $x_{j}x_{k}$ is RED while $x_{j}x_{\ell}$ is BLUE. 
Then $x_{j}$ is RED, as otherwise we obtain a BLUE $Z_3$. However, in such a case $(uwx_{j}x_k,x_{j})$ is a RED $Z_5$, a contradiction.

So, we may assume that $x_2,x_3,x_4,x_5$ are all RED, and any edges in the $K_5$ incident to $u$ are all RED. If any edge incident to $x_2$ but not $u$, say $x_2x_3$ is RED, then as $x_2x_3\in R'$, we obtain a $(ux_2x_3, x_2x_3)$, a RED generalised clique $Z_5$, a contradiction.
Thus $x_2x_3, x_2x_4,x_2x_5$ are all BLUE. Then all three edges $x_3x_4, x_3x_5, x_4x_5$ must be all RED, otherwise we obtain a BLUE $K_3$ together with $x_2$. However, then $(ux_3x4_x5, x_3)$ is a RED generalised clique $Z_5$, a contradiction. Thus, all vertices of $R$ are RED.

	Since each $U_i$ is monochromatic in RED. By the stability for 2-coloured Tur\'an problem, $R[U_1,\ldots,U_5]$ is close to the blow-up $(K_3,K_3)$-free 2-edge-colouring of $K_5$. This almost gives all we need, except that the corresponding partition could have low-degree vertices. Let $X_1'\cup\ldots \cup X_5'$ be the corresponding partition of $V(G)$, i.e.~$X_i':=\cup_{k\in U_i}V_k$.
	By using argument similar to the proof of Claim~\ref{cl-refine-Vi}, 
	we can obtain sets $X_1,\dots, X_5$, such that for each $k\neq j\in [5]$, $d(v,X_k)\ge\delta(G)-3\max|X_i|-n/20\ge n/10$. As $R$ is close to $T_5(|R|)$, $X_i$ and $X_i'$ only differ on $o(n)$ vertices, proving \ref{A'1}. Also as each $U_i$ is monochromatic in RED, we see that the RED graph induced on each $X_i'$, hence also on $X_i$, has $o(n)$ independence number, proving \ref{A'2}. Both \ref{A'3} and \ref{A'4} follow from the fact that that the colour pattern of $R[U_1,\ldots,U_5]$ is close to the blow-up $(K_3,K_3)$-free 2-edge-colouring of $K_5$. 	
\end{proof}
We now show how Lemma~\ref{thm-K3K5-stab} implies the desired upper bound on $\varrho(K_3,K_5,\delta)$.

Assume that $e(G)> (\frac{2}{5}+ \frac{\delta}{2})n^2$.
By applying Lemma~\ref{lem: min degree} with $G,1/2,\delta$ playing the roles of $G,d,\epsilon$, respectively, to obtain an $n'$-vertex graph $G'$ with $n'\geq \delta^{1/2} n/2$ satisfying 
$\delta(G')\geq  (\frac{2}{5} + \frac{\delta'}{2})n'^2$
and $\alpha(G')\leq \alpha(G) = \delta n = \delta' n'$, where $\delta':= \delta n/n' \in [\delta, \delta^{1/3})$.

As $G$ is $(K_3,K_5)$-free, $G'\subseteq G$ is also $(K_3,K_5)$-free and there exists an $(K_3,K_5)$-free $2$ edge-colouring $\phi$ of $G'$.
As $1/n \ll \delta$ and $n'\geq \delta^{1/2} n/2$ and $\delta' < 10^{-6}$, we apply Lemma~\ref{thm-K3K5-stab} to obtain a partition  $X_1\cup \ldots\cup X_5$ of $V(G')$ satisfying \ref{A'1}--\ref{A'4} . We first observe that each $G'[X_i]$ is monochromatic in colour $2$. Indeed, if there is an edge $uv\in E(G'_1[X_i])$, then by \ref{A'3}, we get $K_3\subseteq G_1$, a contradiction.

Suppose that $T\in{X_i\choose 3}$ induces a triangle in $X_i$. By the observation above, $T$ is monochromatic in colour $2$. By \ref{A'1}, \ref{A'2} and \ref{A'4}, we have $|N_{G'_2}(T,X_{i+1})|\ge n/6>\alpha(G'_2[X_{i+1}])$, implying a copy of $K_5$ in $G'_2$, a contradiction.

Hence for each $i\in [5]$, the graph $G'[X_i]$ is triangle-free, which would imply that $e(G')\le \frac{2n'^2}{5}+\frac{\delta n'^2}{2}$ using the calculation similar to~\eqref{eq-K3-free-cal}, a contradiction.
Therefore, we conclude $e(G)\leq \frac{2n^2}{5} + \frac{\delta n^2}{2}$.

\subsection{A missing case in the weak stability in~\cite{EHSSSz}}
We will use the notation in~\cite{EHSSSz}. In their proof, they consider a coloured weighted clique $K_t(w)$.  They assume that there is a BLUE vertex $v$ in the weighted clique $K_t(w)$ (BLUE=colour $1$; while RED=colour $2$). However, it could be that all vertices are RED. This missing case can be easily proven as follows. When all vertices are RED, $K_t(w)$ cannot contain an edge-monochromatic $K_3$, implying that $t\le 5$. Also as all vertices are RED, no RED edge can have weight larger than $1/2$. Together with the fact that BLUE edges induces a triangle-free graph, a simple optimisation shows that maximum is achieved when $t=3$.


\end{document}